\documentclass[a4paper]{article}
\usepackage{amsmath,amsthm,amssymb,amscd, multicol}
\usepackage[dvipdfmx]{graphicx}

\numberwithin{equation}{section}

\newtheorem{thm}{Theorem}[section]
\newtheorem{lem}[thm]{Lemma}
\newtheorem{prop}[thm]{Proposition}
\newtheorem{cor}[thm]{Corollary}
\newtheorem{remark}[thm]{Remark}

\theoremstyle{definition}
\newtheorem{definition}[thm]{Definition}
\newtheorem{example}[thm]{Example}

\newenvironment{rem}{\begin{remark}\rm}{\end{remark}}

\def\address#1#2{\begingroup
\noindent\parbox[t]{12cm}{%
\small{\scshape\ignorespaces#1}\par\vskip1ex
\noindent\small{\itshape E-mail address}%
\/: #2\par\vskip4ex}\hfill%
\endgroup}%
\makeatother

\title{Cohomogeneity One Coassociative Submanifolds 
in the Bundle of Anti-self-dual 2-forms over the 4-sphere}
\author{Kotaro Kawai
\footnote{The author is supported by Grant-in-Aid for JSPS fellows (26-7067).}}
\date{}

\begin{document}

\maketitle

\begin{abstract}

Coassociative submanifolds are 4-dimensional calibrated submanifolds in $G_{2}$-manifolds. 
In this paper, we construct explicit examples of coassociative submanifolds 
in $\Lambda^{2}_{-} S^{4}$, which is the complete $G_{2}$-manifold 
constructed by Bryant and Salamon. 
Classifying the Lie groups which have 3- or 4-dimensional orbits, 
we show that the only homogeneous coassociative submanifold is the
zero section of $\Lambda^{2}_{-} S^{4}$ up to the automorphisms 
and construct many cohomogeneity one examples explicitly. 
In particular, we obtain examples of non-compact coassociative submanifolds 
with conical singularities and their desingularizations. 
\end{abstract}

\section{Introduction}

In 1996, Strominger, Yau and Zaslow  \cite{SYZ} 
presented a conjecture 
explaining mirror symmetry of compact Calabi-Yau 3-folds 
in terms of 
dual fibrations by special Lagrangian 3-tori, including singular fibers. 
Analogously, fibrations of coassociative 4-folds in compact $G_{2}$-manifolds 
are expected to play the same role 
as special Lagrangian fibrations in Calabi-Yau manifolds. 
In this paper, we focus on the construction of 
coassociative 4-folds in a non-compact $G_{2}$-manifold. 
By constructing these examples, 
we will gain a greater understanding of coassociative geometry and 
local models for coassociative submanifolds in compact $G_{2}$-manifolds.

In $\mathbb{R}^{7}$, 
Harvey and Lawson gave ${\rm SU}(2)$-invariant 
coassociative submanifolds in their pioneering paper \cite{Harvey Lawson}.
Lotay \cite{Lotay1, Lotay2}
constructed 2-ruled examples and  ones with the $T^{2} \times \mathbb{R}_{>0}$ symmetries 
using evolution equations. 
Fox \cite{Fox} obtained a family of non-2-ruled, non-conical examples 
from a 2-ruled coassociative cone. 
Ionel, Karigiannis and Min-Oo \cite{Ionel Kari Min} 
gave examples in $\Lambda^{2}_{-} \mathbb{R}^{4} \cong \mathbb{R}^{7}$, 
which are the total spaces of certain rank 2 subbundles 
over immersed surfaces in $\mathbb{R}^{4}$. 
Karigiannis and Leung \cite{Kari_Leung}
generalized this method by twisting the bundles by a special section 
of a complementary bundle.
Karigiannis and Min-Oo \cite{Kari_Min} applied the method in \cite{Ionel Kari Min} 
to $\Lambda^{2}_{-} S^{4}$ and 
$\Lambda^{2}_{-} \mathbb{C}P^{2}$
and obtained some examples. 
Here, $\Lambda^{2}_{-} S^{4}$ and $\Lambda^{2}_{-} \mathbb{C}P^{2}$ admit  
complete $G_{2}$-metrics constructed by Bryant and Salamon \cite{BryantSalamon}. 

In this paper, 
we focus on the case of $\Lambda^{2}_{-} S^{4}$ 
and construct many explicit examples of coassociative submanifolds in $\Lambda^{2}_{-} S^{4}$. 
There exists a family of 
torsion-free 
$G_{2}$-structures $\{ (\varphi_{\lambda}, g_{\lambda}) \}_{\lambda > 0}$
on $\Lambda^{2}_{-} S^{4}$ (Proposition \ref{BS_G2str}). 
For each $\lambda > 0$, 
the automorphism group of $(\Lambda^{2}_{-} S^{4}, \varphi_{\lambda}, g_{\lambda})$ 
is ${\rm SO}(5)$
acting on $\Lambda^{2}_{-} S^{4}$ by the lift of the standard action on $S^{4}$ (\cite{Sal book}). 

First, by classifying the Lie subgroups of ${\rm SO}(5)$ 
which have 4-dimensional orbits in $\Lambda^{2}_{-} S^{4}$, 
we obtain the following result.

\begin{thm} \label{homog coasso}
Let  $\{ (\varphi_{\lambda}, g_{\lambda}) \}_{\lambda > 0}$ be 
the family of torsion-free $G_{2}$-structures on $\Lambda^{2}_{-} S^{4}$ in Proposition \ref{BS_G2str}. 
For each $\lambda > 0$, 
every homogeneous coassociative submanifold in 
 $(\Lambda^{2}_{-} S^{4}, \varphi_{\lambda}, g_{\lambda})$ 
is congruent under the action of ${\rm SO}(5)$ 
to the zero section $S^{4} \subset \Lambda^{2}_{-} S^{4}$.
\end{thm}

Next, we prove that the Lie subgroup of ${\rm SO}(5)$ 
which have 3-dimensional orbits in $\Lambda^{2}_{-} S^{4}$ is 
one of the following (Proposition \ref{Lie grp 3 4-dim orbit}). 
\begin{align*}
{\rm SO}(4) &= {\rm SO}(4) \times \{ 1 \},  \qquad
& &{\rm SO}(3) \times {\rm SO}(2), \qquad
{\rm U}(2), \ {\rm SU}(2) \subset {\rm SO}(4) \times \{ 1 \},  \\
{\rm SO}(3) &= {\rm SO}(3) \times \{ I_{2} \}, \qquad
& &{\rm SO}(3) \mbox{ acting irreducibly on } \mathbb{R}^{5}.
\end{align*}
We derive O.D.E.s 
which give coassociative submanifolds 
by the cohomogeneity one method 
of Hsiang and Lawson \cite{Hsiang_Lawson} 
in each case. 
In many cases, O.D.E.s are solved explicitly and we obtain the following new examples.

Let $(x_{1}, x_{2}, x_{3}, x_{4}, x_{5})$ be standard coordinates of $\mathbb{R}^{5}$ 
and regard $S^{4}$ as the unit sphere in $\mathbb{R}^{5}$. 
Let $(a_{1}, a_{2}, a_{3})$ be the local fiber coordinates  of $\Lambda^{2}_{-} S^{4}$
by choosing a local frame for $\Lambda^{2}_{-} S^{4}$ 
as in Section \ref{second local frame}.

\begin{thm}[Case of ${\rm SO}(3) \times {\rm SO}(2)$] \label{SO3SO2 thm}
Let ${\rm SO}(3) \times {\rm SO}(2)$ act on $\Lambda^{2}_{-} S^{4}$ 
by the lift of the standard ${\rm SO}(3) \times {\rm SO}(2)$-action on $S^{4}$. 
For any $C \in \mathbb{R}$, set 
\begin{align*}
M_{C} = 
{\rm SO}(3) \times {\rm SO}(2) \cdot 
\left \{ 
\left( {}^t\! (x_{1}, 0, 0, \sqrt{1-x_{1}^{2}}, 0), {}^t\! (a_{1}, 0, 0) \right); 
\begin{array}{c}
G(a_{1} ,x_{1}) = C, \\
a_{1} \in \mathbb{R}, 0 < x_{1} \leq 1
\end{array}
\right \}, 
\end{align*}
where $G(a_{1} ,x_{1})$ is defined in (\ref{def of G}). 
Then $M_{C}$ is coassociative and 
it is homeomorphic to 
\begin{align*}
\left\{
\begin{array}{ll}
(S^{2} \times \mathbb{R}^{2}) \sqcup (S^{2} \times S^{1} \times \mathbb{R}_{> 0}) & 
\qquad \mbox{for }\  C \neq 0, \\
S^{4} \sqcup (S^{2} \times S^{1} \times \mathbb{R}_{> 0}) \sqcup (S^{2} \times S^{1} \times \mathbb{R}_{> 0}) &
\qquad \mbox{for }\  C=0,
\end{array}
\right.
\end{align*}
where $S^{4}$ is the zero section of $\Lambda^{2}_{-} S^{4}$. 
\end{thm}

\begin{thm} [Case of ${\rm SU}(2) \subset {\rm SO}(4) \times \{ 1 \}$] \label{SU2 thm}
Let ${\rm SU}(2)$ act on $\Lambda^{2}_{-} S^{4}$ 
by the lift of the standard action of ${\rm SU}(2) \subset {\rm SO}(4) \times \{ 1 \}$ on $S^{4}$. 
For any $C \in \mathbb{R}$ and 
$v \in S^{2} \subset \mathbb{R}^{3}$, set 
\begin{align*}
M_{C, v} := {\rm SU}(2) \cdot \left \{ 
\left({}^t\! (\sqrt{1-x_{5}^{2}}, 0, 0, 0, x_{5}), r v \right) ; 
                                 F(r, x_{5}) = C, r\geq 0, -1 \leq x_{5} \leq 1 \right \}, 
\end{align*}
where $F(r, x_{5})$ is defined in (\ref{def of F}). 
Then $M_{C, v}$ is coassociative and 
it is homeomorphic to 
\begin{align*}
\left \{
\begin{array}{ll}
\mathbb{R}^{4} & \qquad \mbox{for }\ C>0, \\
S^{4} \sqcup (S^{3} \times \mathbb{R}_{>0}) & \qquad \mbox{for }\ C=0, \\
\mathcal{O}_{\mathbb{C} P^{1}} (-1) & \qquad \mbox{for }\  C<0,
\end{array}
\right.
\end{align*}
where $S^{4}$ is the zero section of $\Lambda^{2}_{-} S^{4}$
and $\mathcal{O}_{\mathbb{C} P^{1}} (-1)$ is 
the tautological line bundle over $\mathbb{C} P^{1} \cong S^{2}$. 
\end{thm}

By using the stereographic local coordinates, 
the ${\rm SU}(2)$-action is described as 
in the case of $\mathbb{R}^{7}$. See (\ref{SU2 action stereo graphic coor}). 
In this sense, 
the above example is an analogue of 
an ${\rm SU}(2)$-invariant coassociative submanifold in $\mathbb{R}^{7}$ 
given by Harvey and Lawson \cite{Harvey Lawson}.

\begin{thm}[Case of ${\rm SO}(3) = {\rm SO}(3) \times \{ I_{2} \}$] \label{SO3 thm}
Let ${\rm SO}(3)$ act on $\Lambda^{2}_{-} S^{4}$ 
by the lift of the standard ${\rm SO}(3)  = {\rm SO}(3) \times \{ I_{2} \}$-action on $S^{4}$. 
For $C, D \geq 0$ and $E \in \mathbb{R}$, set 
\begin{align*}
M_{C, D, E} = 
{\rm SO}(3) \cdot 
\left \{ 
({}^t\! (x_{1}, 0, 0, x_{4}, x_{5}), {}^t\! (a_{1}, a_{2}, 0)); 
\begin{array}{c}
x_{4}^{4} (\lambda + r^{2}) = C,\\
x_{5}^{4} (\lambda + r^{2}) = D,\\
a_{1} x_{1} = E
\end{array}
\right \}. 
\end{align*}
Then $M_{C, D, E}$ is coassociative and 
the topology of $M_{C, D, E}$ 
is given in Lemma \ref{topology of M CDE}. 
In particular, we obtain examples of non-compact coassociative submanifolds 
with conical singularities 
for 
$(C, D) \neq (0, 0), E = 0, \sqrt{C} + \sqrt{D} = \sqrt{\lambda}$ 
and their desingularizations. 
\end{thm}

\begin{thm}[Case of irreducible ${\rm SO}(3)$] \label{irr SO3 thm}
Let ${\rm SO}(3)$ act on $\Lambda^{2}_{-} S^{4}$ 
by the lift of the irreducible ${\rm SO}(3)$-action on $S^{4}$. 
Let $x_{1}(t), x_{5}(t), a_{1}(t), a_{2}(t), a_{3}(t)$ be 
smooth functions on an open interval $I \subset \mathbb{R}$ 
satisfying $0 \leq x_{1}(t) \leq \sqrt{3}/2, |x_{5}(t)| < 1/2, x_{1}^{2}(t) +x_{5}^{2}(t) = 1$, 
\begin{align*}
4 \left \{
(2 \sqrt{3} x_{1} + 4 x_{5} + 1) \dot{x}_{5} + \sqrt{3} (2 x_{5} -1) \dot{x}_{1} 
\right \} a_{1}
+ 8 x_{1} (- x_{1} + \sqrt{3} x_{5}) \dot{a}_{1} &\\
- (\sqrt{3} x_{1} + x_{5} +1)(1- 2 x_{5}) a_{1} \frac{d}{dt} \log (\lambda + r^{2}) &= 0, \\
4 \left \{
(2 \sqrt{3} x_{1} - 4 x_{5} - 1) \dot{x}_{5} + \sqrt{3} (2 x_{5} -1) \dot{x}_{1} 
\right \} a_{2}
+ 8 x_{1} (x_{1} + \sqrt{3} x_{5}) \dot{a}_{2} &\\
+ (- \sqrt{3} x_{1} + x_{5} + 1)(1- 2 x_{5}) a_{2} \frac{d}{dt} \log (\lambda + r^{2}) &= 0, \\
4 \left \{
 - (x_{5} + 1) \dot{x}_{5} + 3 x_{1} \dot{x}_{1} 
\right \} a_{3} 
+ 2 (x_{1}^{2} - 3 x_{5}^{2}) \dot{a}_{3} &\\
+ (1 + x_{5})(1- 2 x_{5}) a_{3} \frac{d}{dt} \log (\lambda + r^{2}) &= 0, 
\end{align*}
where 
$r^{2}(t) = \sum_{j=1}^{3} a_{j}^{2}(t)$ and 
$\dot{x}_{1} = d x_{1}/ dt$, etc.
Then 
\begin{align*}
{\rm SO}(3) \cdot \left \{ 
({}^t\! (x_{1}(t), 0, 0, 0, x_{5}(t)), {}^t\! (a_{1}(t), a_{2}(t), a_{3}(t))); t \in I \right \}, 
\end{align*}
is a coassociative submanifold invariant under the irreducible ${\rm SO}(3)$-action.
\end{thm}

This paper is organized as follows. 
In Section 2, 
we review the fundamental facts of 
calibrated geometry and 
$G_{2}$ geometry 
and introduce the cohomogeneity one method 
of Hsiang and Lawson \cite{Hsiang_Lawson}. 
In Section 3, we introduce the $G_{2}$-structure on $\Lambda^{2}_{-} S^{4}$
given by Bryant and Salamon \cite{BryantSalamon}. 
In Section 4, we classify the connected closed subgroups of 
${\rm SO}(5)$, which is the automorphism group of 
the $G_{2}$-manifold $\Lambda^{2}_{-} S^{4}$, 
and study their orbits. 
Classifying Lie subgroups which have 3- or 4-dimensional orbits, 
we prove Theorem \ref{homog coasso}. 
In Section 5, 
according to the classification in Section 4, 
we construct cohomogeneity one coassociative submanifolds 
and prove Theorems \ref{SO3SO2 thm}, \ref{SU2 thm}, \ref{SO3 thm} and \ref{irr SO3 thm}.


\noindent{{\bf Acknowledgements}}: 
The author is very grateful to Professor Katsuya Mashimo 
for his useful advice about the representation theory. 
He would like to thank the referees for the careful
reading of an earlier version of this paper and useful comments on it.


\section{Preliminaries}

\begin{definition}
Define a 3-form $\varphi_{0}$ on $\mathbb{R}^{7}$ by
\begin{align}\label{def of G2 str}
\varphi_{0} = e^{123} + e^{145} - e^{167} + e^{246} + e^{257} + e^{347} - e^{356}, 
\end{align}
where  $(e^{1}, \cdots, e^{7})$ is the standard dual basis 
on $\mathbb{R}^{7}$ 
and wedge signs are omitted. 
The stabilizer of $\varphi_{0}$ is the exceptional Lie group $G_{2}$:
\begin{align*} 
G_{2} = \{ g \in GL(7, \mathbb{R})  ;  g^{*}\varphi_{0} = \varphi_{0} \}.
\end{align*}
This is a 14-dimensional compact simply-connected simple Lie group.
\end{definition}

The Lie group $G_{2}$ also fixes 
the standard metric $g_{0} = \sum^7_{i=1} (e^{i})^{2}$, the orientation on $\mathbb{R}^{7}$ 
and the 4-form 
\begin{eqnarray*}
*\varphi_{0} = e^{4567} + e^{2367} - e^{2345} + e^{1357} + e^{1346} + e^{1256} - e^{1247}, 
\end{eqnarray*}
where $*$ means the Hodge dual.
They are uniquely determined by $\varphi_{0}$ via 
\begin{eqnarray}\label{g varphi}
- 6 g_{0}(v_{1}, v_{2}) {\rm vol}_{g_{0}} = i(v_{1})\varphi_{0} \wedge i(v_{2})\varphi_{0} \wedge \varphi_{0}, 
\end{eqnarray}
where ${\rm vol}_{g_{0}}$ is the volume form of $g_{0}$, 
$i( \cdot )$ is the interior product, and $v_{i} \in T(\mathbb{R}^{7})$.

\begin{definition}
Let $Y$ be a 7-dimensional oriented manifold and 
$\varphi$ be a 3-form on $Y$. 
A 3-form $\varphi$ is called a {\bf $G_{2}$-structure} on $Y$ if 
for each $y \in Y$, there exists an oriented isomorphism between $T_{y}Y$ and $\mathbb{R}^{7}$ 
identifying $\varphi_{y}$ with $\varphi_{0}$. 
From (\ref{g varphi}), $\varphi$ induces the metric $g$ on $Y$, 
volume form on $Y$ and $*\varphi \in \Omega^{4}(Y)$.

A  $G_{2}$-structure $\varphi$ is called {\bf torsion-free} if 
$\varphi$ is closed and coclosed: $ d\varphi = d*\varphi = 0$.
We call a triple $(Y, \varphi, g)$ a {\bf $G_{2}$-manifold} if 
$\varphi \in \Omega^{3}(Y)$ is a torsion-free $G_{2}$-structure on $Y$ 
and $g$ is the associated metric. 
\end{definition}

\begin{lem}[\cite{FG}]
Let $(Y, \varphi, g)$ be a manifold with a $G_{2}$-structure. 
Then the holonomy group of $g$ is contained in $G_{2}$ 
if and only if 
$d\varphi = d * \varphi = 0$.
\end{lem}

Recall the notion of a calibration introduced by Harvey and Lawson \cite{Harvey Lawson}. 
\begin{definition}
Let $(Y, g)$ be an $n$-dimensional Riemannian manifold. 
A closed $k$-form $\varphi$ on $Y$, where $1 \leq k \leq n$, 
is called a {\bf calibration} on $Y$ if 
$
\varphi|_{V} \leq {\rm vol}_{V} 
$
for each point $p \in Y$ and 
every oriented $k$-dimensional subspace $V \subset T_{p}Y$. 
We say that an oriented $k$-dimensional submanifold $L$ of $Y$ 
is a 
{\bf calibrated submanifold} 
of $Y$ (or calibrated by $\varphi$) 
if 
$\varphi|_{TL} = {\rm vol}_{L}.$
\end{definition}

There are canonical calibrations on a $G_{2}$-manifold.

\begin{lem}[\cite{Harvey Lawson}]
 Let $(Y, \varphi, g)$ be a $G_{2}$-manifold.
Then 
the $G_{2}$-structure $\varphi$ and 
its Hodge dual $*\varphi$ define calibrations on Y.
\end{lem}

\begin{definition}[\cite{Harvey Lawson}]
An oriented 3-dimensional submanifold is called an {\bf associative submanifold} of $Y$ 
if it is calibrated by $\varphi$. 
An oriented 4-dimensional submanifold is called a {\bf coassociative submanifold} of $Y$ 
if it is calibrated by $*\varphi$. 
\end{definition}

\begin{lem}[\cite{Harvey Lawson}]
If $L \subset Y$ is an oriented 4-dimensional submanifold,  
then
$L$ is a coassociative submanifold of Y up to a possible change of orientation for $L$
if and only if 
$\varphi|_{TL} = 0$.
\end{lem}

This description is often more useful and easier to work with.

\subsection{Cohomogeneity one method} \label{general method}

Let $L$ be a coassociative submanifold of 
a $G_{2}$-manifold $(Y, \varphi, g)$. 
The symmetry group $K$ of $L$ is defined to
be the Lie subgroup of the automorphism group 
which fixes $L$. 
If the principal orbits of $K$ are of codimension one in $L$, 
we call $L$ a {\bf cohomogeneity one} coassociative submanifold.
The action of $K$ on $L$ is called a {\bf cohomogeneity one action}.

Coassociative submanifolds are defined by 
first order nonlinear P.D.E.s, 
which are difficult to solve in general. 
By the cohomogeneity one action of the Lie group, 
we reduce the P.D.E.s of the coassociative condition to nonlinear O.D.E.s, 
which are easier to solve. 
This method was introduced in \cite{Hsiang_Lawson} for minimal submanifolds. 
We give a summary in our coassociative settings 
based on \cite{Hashimoto_Sakai}.

\begin{lem}
Let $(Y, \varphi, g)$ be a $G_{2}$-manifold and 
$G$ be a Lie subgroup of the automorphism group of $(Y, \varphi, g)$. 
Let $\Sigma \subset Y$ be 
a subset which is transverse to the $G$-orbits and satisfies $G \cdot \Sigma = Y$. 
Suppose that $G$ has 3-dimensional orbits on $Y$. 

Then the solution of the first order nonlinear O.D.E.s $\varphi|_{G \cdot {\rm Image}(c)} = 0$, 
where $c : I \rightarrow \Sigma$ is a path and $I \subset \mathbb{R}$ is an open interval, 
gives a $G$-invariant coassociative submanifold $G \cdot {\rm Image}(c)$. 
\end{lem}

Note that there is a similar construction by using evolution equations. 
This method was introduced by Lotay \cite{Lotay2} for 
associative, coassociative and Cayley submanifolds.


\section{Geometry in $\Lambda^{2}_{-} S^{4}$}

\subsection{$G_{2}$-structure on $\Lambda^{2}_{-} S^{4}$}

We introduce the complete metric on the bundle $\Lambda^{2}_{-} S^{4}$ 
of anti-self-dual 2-forms over the 4-sphere $S^{4}$
obtained by 
Bryant and Salamon \cite{BryantSalamon}. 
We also refer to \cite{Kari_Min, Sal G2}. 
Since $\Lambda^{2}_{-} S^{4}$ has a connection induced by the 
Levi Civita connection on $S^{4}$, 
the tangent space $T_{\omega} (\Lambda^{2}_{-} S^{4})$ has a canonical splitting 
$T_{\omega} (\Lambda^{2}_{-} S^{4}) \cong \mathcal{H}_{\omega} \oplus \mathcal{V}_{\omega}$ 
into horizontal and vertical subspaces 
for each $\omega \in \Lambda^{2}_{-} S^{4}$.

\begin{prop} [Bryant and Salamon \cite{BryantSalamon}]  \label{BS_G2str}
For $\lambda > 0$, define the 3-form $\varphi_{\lambda} \in \Omega^{3}(\Lambda^{2}_{-} S^{4})$ 
and the metric $g_{\lambda}$ on $\Lambda^{2}_{-} S^{4}$ by 
\begin{eqnarray*}
\varphi_{\lambda} = 2 s_{\lambda} d \tau 
+ \frac{1}{s_{\lambda}^{3}} {\rm vol_{\mathcal{V}}}, \qquad 
g_{\lambda} = 2 s_{\lambda}^{2} g_{\mathcal{H}} 
+ \frac{1}{s_{\lambda}^{2}} g_{\mathcal{V}}, 
\end{eqnarray*}
where 
$s_{\lambda} = ( \lambda + r^{2})^{1/4}$, 
$r$ is the distance function measured by the fiber metric 
induced by that on $S^{4}$, 
$\tau$ is a tautological 2-form and 
${\rm vol_{\mathcal{V}}}$ is the volume form of $g_{\mathcal{V}}$ on the vertical fiber. 

Then for each $\lambda > 0$, $(\Lambda^{2}_{-} S^{4}, \varphi_{\lambda}, g_{\lambda})$ is a $G_{2}$-manifold 
and $g_{\lambda}$ 
is the complete metric with holonomy equal to $G_{2}$.
\end{prop}

\begin{rem}
A complete holonomy $G_{2}$ metric is constructed 
not only on $\Lambda^{2}_{-} S^{4}$ but also on $\Lambda^{2}_{-} \mathbb{C}P^{2}$ 
in \cite{BryantSalamon}. 
Of course, we can also 
apply the method 
in Section \ref{general method}
to $\Lambda^{2}_{-} \mathbb{C}P^{2}$ 
and construct examples 
in theory. 
\end{rem}

By using a local frame, 
$\varphi_{\lambda}$ 
is described 
as follows. 
Let $\{ e^{1}, e^{2}, e^{3}, e^{4} \}$ be a local oriented orthonormal coframe 
with respect to the standard metric and the standard orientation on $S^{4}$. 
Define  
2-forms $\omega_{i}$ on $S^{4}$ by  
\begin{eqnarray*}
\omega_{1} = e^{12} - e^{34}, \qquad
\omega_{2} = e^{13} - e^{42}, \qquad
\omega_{3} = e^{14} - e^{23}. 
\end{eqnarray*}
Then 
$\{ \omega_{1}, \omega_{2}, \omega_{3} \}$ 
is a local oriented coframe of $\Lambda^{2}_{-} S^{4}$, 
which is orthogonal but not normalized to unit length, 
and induces the local fiber coordinates $(a_{1}, a_{2}, a_{3})$ of $\Lambda^{2}_{-} S^{4}$. 
Write $\nabla \omega_{i} = \sum_{j = 1}^{3} \gamma_{i j} \otimes \omega_{j}$, 
where $\nabla$ is the induced connection from 
the Levi-Civita connection of the standard metric on $S^{4}$ 
and $\gamma_{i j}$ is a local 1-form.
Let $\pi : \Lambda^{2}_{-} S^{4} \rightarrow S^{4}$ be the projection. 
Denoting $b_{i} = d a_{i} + \sum_{j=1}^{3} a_{j} \pi^{*} \gamma_{j i}$, we have 
\begin{align*}
r^{2} = \sum_{i =1}^{3} a_{i}^{2},  \qquad
\tau = \sum_{i =1}^{3} a_{i} \pi^{*} \omega_{i}, \qquad
d \tau = \sum_{i = 1}^{3} b_{i} \wedge \pi^{*} \omega_{i},  \qquad
{\rm vol}_{\mathcal{V}} = b_{123}, 
\end{align*}
where $b_{123} = b_{1} \wedge b_{2} \wedge b_{3}$. 
Thus the $G_{2}$-structure $\varphi_{\lambda}$ is described as
\begin{align}\label{def of G2 str S4}
\varphi_{\lambda} = 2 s_{\lambda} 
\sum_{i = 1}^{3} b_{i} \wedge \pi^{*} \omega_{i} 
+ \frac{1}{s_{\lambda}^{3}} b_{123}.
\end{align}

\begin{rem} \label{cone_of_CP3}
For $\lambda = 0$, the metric $g_{0}$ is 
a cone metric on $\Lambda^{2}_{-} S^{4} - \{ \mbox{zero section} \} \cong 
\mathbb{C} P^{3} \times \mathbb{R}_{>0}$. 
The metric $g_{\mathbb{C} P^{3}}$ on $\mathbb{C}P^{3}$
induced from $g_{0}$
is not the standard metric, but a 3-symmetric Einstein, non-K\"{a}hler metric.
The metric $g_{0}$ is not complete because of the singularity at 0, while 
its holonomy group is equal to $G_{2}$. 
\end{rem}

%
%
%

\subsection{Local frames of $\Lambda^{2}_{-} S^{4}$}

We use the following local frames of $\Lambda^{2}_{-} S^{4}$ 
for the convenience of computations. 
\subsubsection{Local frame on $S^{4} - \{ x_{5} = \pm 1 \}$} \label{second local frame}
Define a local oriented orthonormal frame $\{ e_{1}, e_{2}, e_{3}, e_{4} \}$ 
on $S^{4} - \{ x_{5} = \pm 1 \}$ by 
\begin{align*} 
(e_{1}, e_{2}, e_{3}, e_{4})
= 
\frac{1}{\sqrt{1-x_{5}^{2}}}
\left(
\left(
\begin{array}{c}
-x_{2} \\
x_{1} \\
-x_{4} \\
x_{3}\\
0
\end{array} 
\right),
\left(
\begin{array}{c}
-x_{3} \\
x_{4} \\
x_{1} \\
-x_{2}\\
0
\end{array} 
\right),
\left(
\begin{array}{c}
-x_{4} \\
-x_{3} \\
x_{2} \\
x_{1}\\
0
\end{array} 
\right),
\left(
\begin{array}{c}
-x_{1} x_{5} \\
-x_{2} x_{5} \\
-x_{3} x_{5}\\
-x_{4} x_{5}\\
1- x_{5}^{2}
\end{array} 
\right)
\right).
\end{align*}

Let $\{ e^{i} \}$ be the dual coframe of $\{ e_{i} \}$. 
Set the local orthogonal trivialization 
$\{ \omega_{1}, \omega_{2}, \omega_{3} \} = \{ e^{12}-e^{34}, e^{13}-e^{42}, e^{14}-e^{23} \}$ 
of $\Lambda^{2}_{-} S^{4}$ 
and denote by  $(a_{1}, a_{2}, a_{3})$ local fiber coordinates 
with respect to $\{ \omega_{1}, \omega_{2}, \omega_{3} \}$.
Recall 1-forms $\gamma_{i j}$ and $b_{i}$ are defined by 
$\nabla \omega_{i} = \sum_{j=1}^{3} \gamma_{i j} \otimes \omega_{j}$, 
$b_{i} = d a_{i} + \sum_{j=1}^{3} a_{j} \pi^{*} \gamma_{j i}$. 
Denote by $\nabla^{S^{4}}$ the Levi-Civita connection 
of the standard metric on $S^{4}$.
Then we see the following by a straightforward computation.

\begin{lem}
\begin{align*}
(\nabla_{e_{i}}^{S^{4}} e^{j}) &= 
\frac{1}{\sqrt{1-x_{5}^{2}}}
\left(
\begin{array}{cccc}
x_{5} e^{4} & -e^{3}       & e^{2}        & -x_{5} e^{1} \\
e^{3}        & x_{5} e^{4}  & -e^{1}      & -x_{5} e^{2}\\
-e^{2}      & e^{1}         & x_{5} e^{4} & -x_{5} e^{3}\\
0            &0               &  0           &  0\\
\end{array} 
\right), \\
( 
\gamma_{i j} 
)
&=
\frac{1 + x_{5} }{\sqrt{1-x_{5}^{2}}}
\left( 
\begin{array}{ccc}
0     &   -e^{1}  & e^{2}\\
e^{1} &     0     & e^{3} \\
-e^{2} & -e^{3} &0
\end{array}
\right), \\
\left( 
\begin{array}{c}
b_{1}\\
b_{2} \\
b_{3}
\end{array}
\right) 
&=
\left( 
\begin{array}{c}
da_{1}\\
da_{2} \\
da_{3}
\end{array}
\right) 
+
\frac{1+x_{5}}{\sqrt{1-x_{5}^{2}}} 
\left( 
\begin{array}{c}
a_{2} e^{1} - a_{3} e^{2} \\
-a_{1} e^{1} - a_{3} e^{3} \\
a_{1} e^{2} + a_{2} e^{3} 
\end{array}
\right).
\end{align*}
\end{lem}

\subsubsection{Frame at $\underline{p_{0}} = {}^t\! (0, 0, 0, 0, \pm 1)$} \label{frame at pole}

Set an oriented orthonormal basis $\{ f_{1}, f_{2}, f_{3}, f_{4} \}$ of $T_{\underline{p_{0}}} S^{4}$ by 
\begin{align*}
f_{1} &= {}^t\! (1, 0, 0, 0, 0), \qquad
f_{2} = {}^t\! (0, 1, 0, 0, 0), \\
f_{3} &= {}^t\! (0, 0, 1, 0, 0), \qquad
f_{4} = {}^t\! (0, 0, 0, \pm 1, 0).
\end{align*}
Note that the induced orientation on $T_{{}^t\! (0, 0, 0, 0,1)} S^{4}$ 
is opposite to that of $T_{{}^t\! (0, 0, 0, 0,-1)} S^{4}$.
Let $\{ f^{i} \}$ be the dual coframe of $\{ f_{i} \}$. 
Then a basis 
$\{ \Omega_{1}, \Omega_{2}, \Omega_{3} \} = \{ f^{12}-f^{34}, f^{13}-f^{42}, f^{14}-f^{23} \}$ 
of $\Lambda^{2}_{-} S^{4}|_{\underline{p_{0}}}$ 
gives fiber coordinates $(A_{1}, A_{2}, A_{3})$ of $\Lambda^{2}_{-} S^{4}|_{\underline{p_{0}}}$.

\section{Orbits of closed Lie subgroups of ${\rm SO}(5)$} \label{section orbits of Lie subgrp}

For each $\lambda > 0$, 
the automorphism group of $(\Lambda^{2}_{-} S^{4}, \varphi_{\lambda}, g_{\lambda})$ is 
${\rm SO}(5)$ 
acting on $\Lambda^{2}_{-} S^{4}$ as the lift of the standard action on $S^{4}$ (\cite{Sal book}). 
We study the Lie subgroups of ${\rm SO}(5)$ to 
obtain homogeneous and cohomogeneity one coassociative submanifolds. 
By the classification of compact Lie groups, we obtain the following.

\begin{lem} \label{Lie subgrp of SO5}
The $k$-dimensional connected closed Lie subgroup of ${\rm SO}(5)$, 
where $3 \leq k \leq 10$, 
is one of the following. 
\begin{align*}
\begin{array}{ll}
{\rm SO}(5),\\
{\rm SO}(4) = {\rm SO}(4) \times \{ 1 \}, \\
{\rm SO}(3) \times {\rm SO}(2), \\
{\rm U}(2) \subset {\rm SO}(4) \times \{ 1 \},\\
\end{array}
\hspace{1cm}
\begin{array}{ll}
{\rm SU}(2) \subset {\rm SO}(4) \times \{ 1 \}, \\
{\rm SO}(3) =  {\rm SO}(3) \times \{ I_{2} \},\\
{\rm SO}(3) \mbox{ acting irreducibly on } \mathbb{R}^{5}. 
\end{array}
\end{align*}
\end{lem}

The proof is given in Appendix \ref{proof of classification subgrp SO(5)}. 
According to Lemma \ref{Lie subgrp of SO5}, 
we study the orbits on $\Lambda^{2}_{-} S^{4}$ of Lie subgroups of ${\rm SO}(5)$ above. 


\subsection{${\rm SO}(4) = {\rm SO}(4) \times \{ 1 \}$ and ${\rm SO}(5)$-actions} \label{section SO4}

In this subsection, 
We consider both the ${\rm SO}(4) = {\rm SO}(4) \times \{ 1 \}$ 
and the ${\rm SO}(5)$-orbits.

\begin{lem}[Orbits of the ${\rm SO}(4)$-action] \label{orbit SO4}
By the ${\rm SO}(4)$-action, 
any point in $\Lambda^{2}_{-} S^{4}$ is mapped to a point 
in the fiber of 
$\underline{p_{0}} = {}^t\! (x_{1}, 0, 0, 0, x_{5})$
where $x_{1} \geq 0$. 
The  ${\rm SO}(4)$-orbit through 
$p_{0} \in \Lambda^{2}_{-} S^{4}|_{\underline{p_{0}}}$ 
is diffeomorphic to 
\begin{align*}
\left\{
\begin{array}{ll}
 {\rm SO}(4) / {\rm SO}(2) & \qquad \mbox{for }\  x_{1} > 0, p_{0} \neq 0, \\ 
 S^{3}                             &  \qquad \mbox{for }\  x_{1} > 0, p_{0} = 0, \\ 
S^{2}                  &  \qquad \mbox{for }\  x_{5}= \pm 1, p_{0} \neq 0, \\
*                       &  \qquad \mbox{for }\  x_{5}= \pm 1, p_{0} = 0.
\end{array}
\right.
\end{align*}
\end{lem}

\begin{cor} \label{orbit SO5}
Let $\mathcal{O}$ be an ${\rm SO}(5)$-orbit. 
Then $\dim \mathcal{O} \leq 4$ if and only if 
$\mathcal{O}$ is the zero section $S^{4}$. 
\end{cor}

\begin{proof}[Proof of Lemma \ref{orbit SO4}]
It is obvious that 
any point in $\Lambda^{2}_{-} S^{4}$ is congruent to a point 
in the fiber of 
$\underline{p_{0}} = {}^t\! (x_{1}, 0, 0, 0, x_{5})$, where $x_{1} \geq 0$, 
by the ${\rm SO}(4) = {\rm SO}(4) \times \{ 1 \}$-action.

Suppose that $x_{1}>0$. 
Since the stabilizer of the ${\rm SO}(4)$-action on $S^{4}$ at $\underline{p_{0}}$ 
is 
${\rm SO}(3) = \{ 1 \} \times {\rm SO}(3) \times \{ 1 \} \subset {\rm SO}(5)$, 
we consider this ${\rm SO}(3)$-action on $\Lambda^{2}_{-} S^{4}|_{\underline{p_{0}}}.$
Use the notation in Section \ref{second local frame}. 
Since $x_{2} = x_{3} = x_{4} = 0$, 
the action of $g = (g_{i j}) \in {\rm SO}(3) = \{ 1 \} \times {\rm SO}(3) \times \{ 1 \}$ is given by 
\begin{align*}
g_{*} e_{i}  = \sum_{j=1}^{3} g_{j i} e_{j} \ \mbox{ for }  i=1,2,3, 
\qquad 
g_{*} e_{4} = e_{4}
\end{align*}
at $\underline{p_{0}}$. 
Then the induced action of $g = (g_{i j}) \in {\rm SO}(3)$ on $\Lambda^{2}_{-} S^{4}|_{\underline{p_{0}}}$ 
is described as 
\begin{align*}
\left( 
\begin{array}{c}
a_{1} \\
a_{2} \\
a_{3} \\
\end{array} 
\right)
\mapsto
\left( 
\begin{array}{ccc}
g_{33} & -g_{32} & -g_{31}\\
-g_{23} & g_{22} & g_{21}\\
-g_{13} & g_{12} & g_{11}\\
\end{array} 
\right)
\left( 
\begin{array}{c}
a_{1} \\
a_{2} \\
a_{3} \\
\end{array} 
\right).
\end{align*}
Thus 
the stabilizer of the ${\rm SO}(4)$-action on $\Lambda^{2}_{-} S^{4}$ at 
$p_{0} = ( {}^t\! (x_{1}, 0, 0, 0, x_{5}), {}^t\! (a_{1}, a_{2}, a_{3}))$  
is 
${\rm SO}(2)$ when ${}^t\! (a_{1}, a_{2}, a_{3}) \neq 0$. 
It is ${\rm SO}(3)$ when ${}^t\! (a_{1}, a_{2}, a_{3}) = 0$.

Next, suppose that $x_{5}= \pm 1$. 
Then the stabilizer of the ${\rm SO}(4)$-action on $S^{4}$ at 
$\underline{p_{0}} = {}^t\! (0, 0, 0, 0, \pm 1)$ 
is ${\rm SO}(4)$.  
By using the frame in Section \ref{frame at pole}, 
the induced action of 
${\rm SO}(4)$ on $\Lambda^{2}_{-} S^{4}|_{\underline{p_{0}}}$ 
is equivalent to 
that of  ${\rm SO}(4) = ({\rm Sp}(1) \times {\rm Sp}(1)) / \mathbb{Z}_{2} $ on 
$\Lambda^{2}_{\mp} \mathbb{R}^{4} = \mathbb{R}^{3} = {\rm Im} \mathbb{H}$, 
which is described as 
\begin{align*}
[(p,q)] \cdot a &= q a \overline{q}  \qquad \mbox{ if } \underline{p_{0}} = {}^t\! (0, 0, 0, 0,1), \\
[(p,q)] \cdot a &= p a \overline{p}  \qquad \mbox{ if } \underline{p_{0}} = {}^t\! (0, 0, 0, 0,-1). 
\end{align*}
This is the standard action of ${\rm Sp}(1)/ \mathbb{Z}_{2} = {\rm SO}(3)$ on $\mathbb{R}^{3}$, 
and hence we obtain the lemma.
\end{proof}

\begin{proof}[Proof of Corollary \ref{orbit SO5}]
It is obvious that 
any point in $\Lambda^{2}_{-} S^{4}$ is congruent to a point 
in the fiber of 
$\underline{p_{0}} = {}^t\! (1, 0, 0, 0, 0)$ by the ${\rm SO}(5)$-action. 
By Lemma \ref {orbit SO4}, 
the subgroup ${\rm SO}(4) \subset {\rm SO}(5)$ has 
5-dimensional orbits on each point of 
$\Lambda^{2}_{-} S^{4}|_{\underline{p_{0}}} - \{ 0 \}$. 
Hence $\mathcal{O}$ must be the zero section $S^{4}$. 
\end{proof}


\subsection{${\rm SO}(3) \times {\rm SO}(2)$-action} \label{section SO3 SO2}

Use the notation in Section \ref{second local frame}. 

\begin{lem}[Orbits of the ${\rm SO}(3) \times {\rm SO}(2)$-action] \label{orbit SO3 SO2}

By the ${\rm SO}(3) \times {\rm SO}(2)$-action, 
any point in $\Lambda^{2}_{-} S^{4}$ is mapped to a point 
in the fiber of 
$\underline{p_{0}} = {}^t\! (x_{1}, 0, 0, x_{4}, 0) \in S^{4}$, 
where $x_{1}, x_{4} \geq 0$.
The ${\rm SO}(3) \times {\rm SO}(2)$-orbit through 
$p_{0} =  ({}^t\! (x_{1}, 0, 0, x_{4}, 0), {}^t\! (a_{1}, a_{2}, a_{3})) 
\in \Lambda^{2}_{-} S^{4}|_{\underline{p_{0}}}$ 
is diffeomorphic to 
\begin{align*}
\left\{
\begin{array}{ll}
 {\rm SO}(3) \times {\rm SO}(2) & \qquad \mbox{for }\  0< x_{1} < 1, (a_{2}, a_{3}) \neq 0, \\ 
 S^{2} \times S^{1}                   & \qquad \mbox{for }\  0< x_{1} < 1, (a_{2}, a_{3}) = 0, \\ 
({\rm SO}(3) \times {\rm SO}(2))/ {\rm SO}(2) & \qquad \mbox{for }\ x_{1}=1, (a_{2}, a_{3}) \neq 0, \\
S^{2}                  & \qquad \mbox{for }\  x_{1}=1, (a_{2}, a_{3}) = 0,\\
S^{2} \times S^{1} & \qquad \mbox{for }\  x_{1}=0, (a_{1}, a_{2}, a_{3}) \neq 0, \\
S^{1}                   & \qquad \mbox{for }\  x_{1}=0, (a_{1}, a_{2}, a_{3}) = 0.
\end{array}
\right.
\end{align*}
When $x_{1}=1, (a_{2}, a_{3}) \neq 0$, 
the dividing group ${\rm SO}(2)$ is identified with 
\begin{align*}
\left \{ 
\left(
\left( 
\begin{array}{cc}
1 &    \\
   & h \\
\end{array}
\right), 
h
\right) 
\in {\rm SO}(3) \times {\rm SO}(2)
;
h \in {\rm SO}(2)
\right \}.
\end{align*}
\end{lem}

\begin{proof}
A direct computation gives the following descriptions. 
When $x_{5} \neq \pm 1$, 
the action of $g = (g_{ij}) \in {\rm SO}(3) = {\rm SO}(3) \times \{ I_{2} \}$ is given by 
\begin{align}\label{SO3 times 1 action}
&g \cdot
\left(
{}^t\! ( x_{1}, 0, 0, x_{4}, x_{5}), 
{}^t\! (a_{1}, a_{2}, a_{3})
\right) \\
= &
\left(
{}^t\! \left( 
g_{11} x_{1}, 
g_{21} x_{1},
g_{31} x_{1},
x_{4}, x_{5}
\right), 
\left( 
\begin{array}{ccc}
g_{11}  & g_{12}   & -g_{13} \\
g_{21}  & g_{22}   & -g_{23} \\
-g_{31} & -g_{32} & g_{33} \\
\end{array}
\right)
{}^t\! (a_{1}, a_{2}, a_{3})
\right). \nonumber
\end{align}
When  $x_{4} \neq \pm 1$, 
the action of 
$h =  
\left( 
\begin{array}{cc}
\cos \alpha & -\sin \alpha\\
\sin \alpha  & \cos \alpha
\end{array}
\right) \in {\rm SO}(2) = \{ I_{3} \} \times  {\rm SO}(2)$ is given by 
\begin{align} \label{1 times SO2 action}
h \cdot
\left(
{}^t\! (x_{1}, 0, 0, x_{4}, 0), 
{}^t\! (a_{1}, a_{2}, a_{3})
\right)
= 
\left(
{}^t\! ( x_{1}, 0, 0, x_{4} \cos \alpha, x_{4} \sin \alpha), 
\left( 
\begin{array}{ccc}
1  &   \\
    &  A\\
\end{array}
\right)
{}^t\! ( a_{1}, a_{2}, a_{3})
\right), 
\end{align}
where 
\begin{align*}
A= 
\frac{1}{1-x_{4} \sin \alpha}
\left( 
\begin{array}{cc}
x_{4}^{2}(1 - \cos \alpha) - x_{4} \sin \alpha + \cos \alpha & x_{1} x_{4} (1-\cos \alpha) - x_{1} \sin \alpha \\
-x_{1} x_{4} (1-\cos \alpha) + x_{1} \sin \alpha & x_{4}^{2}(1 - \cos \alpha) - x_{4} \sin \alpha + \cos \alpha 
\end{array}
\right).
\end{align*}

At $\underline{p_{0}} = {}^t\! (0, 0, 0, 1, 0)$, 
set the orthonormal basis $\{ f_{1}, f_{2}, f_{3}, f_{4} \}$ of $T_{\underline{p_{0}}} S^{4}$ by 
$
f_{1} = {}^t\! (0, 0, -1, 0, 0),
f_{2} = {}^t\! (0, 1, 0, 0, 0),
f_{3} = {}^t\! (-1, 0, 0, 0, 0),
f_{4} = {}^t\! (0, 0, 0, 0, 1).
$
Let $\{ f^{i} \}$ be the dual coframe of $\{ f_{i} \}$. 
Then the local trivialization 
$\{ \Omega_{1}, \Omega_{2}, \Omega_{3} \} = \{ f^{12}-f^{34}, f^{13}-f^{42}, f^{14}-f^{23} \}$ 
of $\Lambda^{2}_{-} S^{4}$ 
gives local fiber coordinates $(A_{1}, A_{2}, A_{3})$ of $\Lambda^{2}_{-} S^{4}$.
The action of $g = (g_{ij}) \in {\rm SO}(3) = {\rm SO}(3) \times \{ I_{2} \}$
on $\Lambda^{2}_{-}S^{4}|_{\underline{p_{0}}}$ is given by  
\begin{align*}
\left( 
\begin{array}{c}
A_{1} \\
A_{2} \\
A_{3} \\
\end{array}
\right)
\mapsto
\left( 
\begin{array}{ccc}
g_{11}  & g_{12}   & -g_{13} \\
g_{21}  & g_{22}   & -g_{23} \\
-g_{31} & -g_{32} & g_{33} \\
\end{array}
\right)
\left( 
\begin{array}{c}
A_{1} \\
A_{2} \\
A_{3} \\
\end{array}
\right)
\end{align*}
By these computations, 
we see the lemma 
as in the proof of Lemma \ref{orbit SO4}. 
\end{proof}

Define the basis $\{ E_{i} \}_{1 \leq i \leq 3}$ of $\mathfrak{so}(3)$  by 
\begin{align} \label{basis of so3}
E_{1} = 
\left( 
\begin{array}{ccc}
0 & -1 & 0 \\
1 & 0   & 0 \\
0 & 0   & 0
\end{array}
\right), \qquad
E_{2} = 
\left( 
\begin{array}{ccc}
0 & 0 & -1 \\
0 & 0 & 0   \\
1 & 0 & 0 
\end{array}
\right), \qquad
E_{3} = 
\left( 
\begin{array}{ccc}
0 & 0 & 0 \\
0 & 0 & -1 \\
0 & 1 & 0  
\end{array}
\right), 
\end{align}
and set 
$E_{4} = 
\left( 
 \begin{array}{ccc}
0  & -1 \\
1  & 0 \\
\end{array}
\right) \in \mathfrak{so}(2)$. 
Via the identifications
$\mathfrak{so}(3) = \mathfrak{so}(3) \oplus \{ 0 \}$ and 
$\mathfrak{so}(2) = \{ 0 \} \oplus \mathfrak{so}(2)$, 
$\{ E_{i} \}_{1 \leq i \leq 4}$ form a basis of $\mathfrak{so}(3) \oplus \mathfrak{so}(2)$. 
By (\ref{SO3 times 1 action}) and (\ref{1 times SO2 action}), 
the vector fields $\tilde{E_{i}^{*}}$ on $\Lambda^{2}_{-} S^{4}$ generated by $E_{i}$ 
are described as 
\begin{align*}
\tilde{E_{1}^{*}} &= x_{1} (x_{1} e_{1}   + x_{4} e_{2} ) 
- a_{2} \frac{\partial}{\partial a_{1}} + a_{1} \frac{\partial}{\partial a_{2}}, \\
\tilde{E_{2}^{*}} &= x_{1} (-x_{4} e_{1}   + x_{1} e_{2} ) 
+ a_{3} \frac{\partial}{\partial a_{1}} - a_{1} \frac{\partial}{\partial a_{3}},\\
\tilde{E_{3}^{*}} &=
a_{3} \frac{\partial}{\partial a_{2}} - a_{2} \frac{\partial}{\partial a_{3}},\\
\tilde{E_{4}^{*}} &= 
x_{4} e_{4} + x_{1} \left(
- a_{3} \frac{\partial}{\partial a_{2}} + a_{2} \frac{\partial}{\partial a_{3}} \right), 
\end{align*}
at $p_{0} = \left( {}^t\! (x_{1}, 0, 0, x_{4}, 0), {}^t\! (a_{1}, a_{2}, a_{3}) \right).$
A straightforward computation gives the following.
\begin{lem} \label{data SO3SO2}
At $p_{0} = \left( {}^t\! (x_{1}, 0, 0, x_{4}, 0), {}^t\! (a_{1}, a_{2}, a_{3}) \right)$, we have 
\begin{align*}
\left( 
\pi^{*} \omega_{j} (\tilde{E_{1}^{*}}, \tilde{E_{2}^{*}}) 
\right)
&= (x_{1}^{2}, 0, 0), 
&
\left( 
\pi^{*} \omega_{j} (\tilde{E_{1}^{*}}, \tilde{E_{4}^{*}})
\right) 
&= (0, x_{1} x_{4}^{2}, x_{1}^{2} x_{4}), \\
\left( 
\pi^{*} \omega_{j} (\tilde{E_{1}^{*}}, \tilde{E_{3}^{*}}) 
\right)
&= 0, 
&
\left( 
\pi^{*} \omega_{j} (\tilde{E_{2}^{*}}, \tilde{E_{4}^{*}}) 
\right)
&= (0, x_{1}^{2} x_{4}, -x_{1} x_{4}^{2}), \\
\left( 
\pi^{*} \omega_{j} (\tilde{E_{2}^{*}}, \tilde{E_{3}^{*}}) 
\right)
&= 0, 
&
\left( 
\pi^{*} \omega_{j} (\tilde{E_{3}^{*}}, \tilde{E_{4}^{*}}) 
\right)
&= 0,
\end{align*}
\begin{align*}
\left( b_{i} (\tilde{E_{j}^{*}}) \right) 
=
\left( 
 \begin{array}{cccc}
-x_{4} (a_{2} x_{4} + a_{3} x_{1}) & x_{4} (-a_{2} x_{1} + a_{3} x_{4}) & 0      & 0\\
a_{1} x_{4}^{2}                      & a_{1} x_{1} x_{4}                      & a_{3}  & -a_{3} x_{1} \\
a_{1} x_{1} x_{4}                    & -a_{1} x_{4}^{2}                      & -a_{2}& a_{2} x_{1}\\
\end{array}
\right).
\end{align*}
\end{lem}


\subsection{Action of ${\rm U}(2) \subset {\rm SO}(4) \times \{ 1 \}$} \label{section U2}

Use the notation in Section \ref{second local frame} and \ref{frame at pole}.

\begin{lem}[Orbits of the ${\rm U}(2)$-action] \label{orbit U2}
By the ${\rm U}(2)$-action, 
any point in $\Lambda^{2}_{-} S^{4}$ is mapped to a point 
in the fiber of 
$\underline{p_{0}} = {}^t\! (x_{1}, 0, 0, 0, x_{5})$ 
for some $x_{1} \geq 0$.
The ${\rm U}(2)$-orbit through 
$p_{0} =  ({}^t\! (x_{1}, 0, 0, 0, x_{5}), {}^t\! (a_{1}, a_{2}, a_{3})) 
\in \Lambda^{2}_{-} S^{4}|_{\underline{p_{0}}}$ 
is diffeomorphic to 
\begin{align*}
\left\{
\begin{array}{ll}
{\rm U}(2) & \qquad \mbox{for }\  x_{5} \neq \pm 1, (a_{1}, a_{2}) \neq 0, \\ 
 S^{3}       & \qquad \mbox{for }\  x_{5} \neq \pm 1, (a_{1}, a_{2}) = 0, \\ 
S^{2}        & \qquad \mbox{for }\  x_{5} = 1, p_{0} \neq 0, \\
S^{1}        & \qquad \mbox{for }\  x_{5} = -1, (A_{2}, A_{3}) \neq 0, \\
*             & \qquad \mbox{for }\  x_{5}=1, p_{0} = 0, \mbox{ or }\ x_{5}=-1, (A_{2}, A_{3}) = 0.
\end{array}
\right.
\end{align*}
\end{lem}

\begin{proof}
Suppose that $x_{5} \neq \pm 1$. 
Denoting 
$
z_{1} = x_{1} + i x_{2}, 
z_{2} = x_{3} + i x_{4}
$, 
we have 
\begin{align*} 
(e_{1}, e_{2}, e_{3}, e_{4}) 
&= 
\frac{1}{\sqrt{1-x_{5}^{2}}} 
\left(
\left(
\begin{array}{c}
i z_{1} \\
i z_{2} \\
0
\end{array}
\right), 
\left(
\begin{array}{c}
- \overline{z}_{2} \\
\overline{z}_{1} \\
0
\end{array}
\right), 
\left(
\begin{array}{c}
- i \overline{z}_{2} \\
i \overline{z}_{1} \\
0
\end{array}
\right), 
\left(
\begin{array}{c}
- x_{5} z_{1} \\
- x_{5} z_{2} \\
1-x_{5}^{2}
\end{array}
\right)
\right) \\
&\subset 
\mathbb{C}^{2} \oplus \mathbb{R}. 
\end{align*}
We see that $e_{1}, e_{2}, e_{3}$ and $e_{4}$ are ${\rm SU}(2)$-invariant. 
Namely, $g_{*} e_{i} = e_{i}$ for any $1 \leq i \leq 4$ and $g \in {\rm SU}(2)$. 
Then 
the 2-forms $\omega_{i}$ are all ${\rm SU}(2)$-invariant, 
and hence 
$
g \in {\rm SU}(2)
$
acts on $\Lambda^{2}_{-} S^{4}$ by 
\begin{align} \label{SU2 action on frame}
g \cdot  ({}^t\! (z_{1}, z_{2}, x_{5}), {}^t\! (a_{1}, a_{2}, a_{3}) ) 
= 
( {}^t\!(g {}^t\! (z_{1}, z_{2}), x_{5}),  {}^t\! (a_{1}, a_{2}, a_{3}) ). 
\end{align}
The action of 
$k(\theta) = 
\left( 
\begin{array}{cc}
1 & \\
   & e^{i \theta}
\end{array} 
\right)
\in {\rm U}(2)$, 
where $\theta \in \mathbb{R}$,  
is given by 
\begin{align*}
&\left(
k(\theta)_{*} e_{1}, k(\theta)_{*} e_{2}, k(\theta)_{*} e_{3}, k(\theta)_{*} e_{4}
\right)\\
=&
\left(
e_{1}, 
e_{2} \cos \theta  + e_{3} \sin \theta, 
- e_{2} \sin \theta + e_{3} \cos \theta, 
e_{4}
\right), 
\end{align*}
which induces the action of $k(\theta)$ on $\Lambda^{2}_{-} S^{4}$ 
described as 
\begin{align} \label{U2-SU2 action on frame}
\left( 
\begin{array}{c}
a_{1} \\
a_{2} \\
a_{3} \\
\end{array} 
\right)
\mapsto
\left( 
\begin{array}{ccc}
\cos \theta & -\sin \theta & 0\\
\sin \theta  & \cos \theta  & 0\\
0               & 0                 & 1\\
\end{array} 
\right)
\left( 
\begin{array}{c}
a_{1} \\
a_{2} \\
a_{3} \\
\end{array} 
\right).
\end{align}
Since any element in ${\rm U}(2)$ is described as 
$k(\theta) g$ for some $\theta$ and $g \in {\rm SU}(2)$, 
we see the case $x_{5} \neq \pm 1$.

Next, suppose that $x_{5} = \pm 1$. 
Then the stabilizer of the ${\rm U}(2)$-action on $S^{4}$ at 
$\underline{p_{0}} = {}^t\! (0, 0, 0, 0, \pm 1)$ 
is ${\rm U}(2)$.  
By using the notation in Section \ref{frame at pole}, 
the induced action of 
$
k(\theta) g, 
$
where 
$
\theta \in \mathbb{R}, g \in {\rm SU}(2), 
$
on $\Lambda^{2}_{-} S^{4}|_{\underline{p_{0}}}$ 
is described as 
\begin{align*}
\left( 
\begin{array}{c}
A_{1} \\
A_{2} \\
A_{3} \\
\end{array} 
\right)
\mapsto
\left( 
\begin{array}{ccc}
1 & 0                 & 0\\
0 & \cos \theta & \sin \theta \\
0 & -\sin \theta  & \cos \theta \\
\end{array} 
\right)
\varpi'(g) 
\left( 
\begin{array}{c}
A_{1} \\
A_{2} \\
A_{3} \\
\end{array} 
\right), 
\end{align*}
where $\varpi'$ is a a double covering $\varpi : {\rm SU}(2) \rightarrow {\rm SO}(3)$ 
(resp. a trivial representation) when $x_{5} = 1$ (resp. $x_{5} = -1$). 
This gives the proof in the case $x_{5} = \pm 1$. 
\end{proof}

Note that the double covering  $\varpi : {\rm SU}(2) \rightarrow {\rm SO}(3)$ 
is given by
\begin{align} \label{covering_SU2_SO3}
\varpi \left(
\left( 
\begin{array}{cc}
a & - \overline{b} \\
b & \overline{a}
\end{array} 
\right)
\right)
=
\left( 
\begin{array}{ccc}
|a|^{2} - |b|^{2}                 & 2 {\rm Im}(ab)            & -2 {\rm Re}(ab)\\
-2 {\rm Im}(\overline{a} b) & {\rm Re}(a^{2} + b^{2})  & {\rm Im}(a^{2} + b^{2})\\
 2 {\rm Re}(\overline{a} b) & {\rm Im}(-a^{2} + b^{2}) & {\rm Re}(a^{2} - b^{2}) 
\end{array} 
\right), 
\end{align}
where $a, b \in \mathbb{C}$ such that $|a|^{2} + |b|^{2} = 1$. 

Define the basis $\{ E_{i} \}_{1 \leq i \leq 4}$ of $\mathfrak{u}(2)$  by 
\begin{align} \label{basis of u2}
E_{1} = 
\left( 
\begin{array}{ccc}
0 & 1  \\
-1 & 0  
\end{array}
\right), \qquad
E_{2} = 
\left( 
\begin{array}{ccc}
0 & i  \\
i & 0  
\end{array}
\right), \qquad
E_{3} = 
\left( 
\begin{array}{ccc}
i & 0  \\
0 & -i  
\end{array}
\right), \qquad
E_{4} = 
\left( 
\begin{array}{ccc}
i & 0  \\
0 & i  
\end{array}
\right). 
\end{align}
By (\ref{SU2 action on frame}) and (\ref{U2-SU2 action on frame}), 
the vector fields $\tilde{E_{i}^{*}}$ on $\Lambda^{2}_{-} S^{4}$ generated by $E_{i}$ 
are described as 
\begin{align*}
\tilde{E_{1}^{*}} = - x_{1} e_{2}, \qquad
\tilde{E_{2}^{*}} = x_{1} e_{3}, \qquad
\tilde{E_{3}^{*}} = x_{1} e_{1}, \qquad
\tilde{E_{4}^{*}} = x_{1} e_{1}  
-2 a_{2} \frac{\partial}{\partial a_{1}} + 2 a_{1} \frac{\partial}{\partial a_{2}}
\end{align*}
at $p_{0} = \left( {}^t\! (x_{1}, 0, 0, 0, x_{5}), {}^t\! (a_{1}, a_{2}, a_{3}) \right).$
A straightforward computation gives the following.

\begin{lem}\label{data U2}
At $p_{0} = \left( {}^t\! (x_{1}, 0, 0, 0, x_{5}), {}^t\! (a_{1}, a_{2}, a_{3}) \right)$, we have 
\begin{align*}
\left( 
\pi^{*} \omega_{j} (\tilde{E_{1}^{*}}, \tilde{E_{2}^{*}}) 
\right)
&= (0, 0, x_{1}^{2}), 
&
\left( 
\pi^{*} \omega_{j} (\tilde{E_{1}^{*}}, \tilde{E_{4}^{*}})
\right) 
&= (x_{1}^{2}, 0, 0), \\
\left( 
\pi^{*} \omega_{j} (\tilde{E_{1}^{*}}, \tilde{E_{3}^{*}}) 
\right)
&= (x_{1}^{2}, 0, 0), 
&
\left( 
\pi^{*} \omega_{j} (\tilde{E_{2}^{*}}, \tilde{E_{4}^{*}}) 
\right)
&= (0, -x_{1}^{2}, 0), \\
\left( 
\pi^{*} \omega_{j} (\tilde{E_{2}^{*}}, \tilde{E_{3}^{*}}) 
\right)
&= (0, -x_{1}^{2}, 0), 
&
\left( 
\pi^{*} \omega_{j} (\tilde{E_{3}^{*}}, \tilde{E_{4}^{*}}) 
\right)
&= 0,
\end{align*}
\begin{align*}
\left( b_{i} (\tilde{E_{j}^{*}}) \right) 
=
\left( 
\begin{array}{cccc}
(1+ x_{5}) a_{3}  & 0                       & (1+ x_{5}) a_{2}  & (-1 + x_{5}) a_{2}\\
0                   & - (1+ x_{5}) a_{3}   & -(1+ x_{5}) a_{1} & (1 - x_{5}) a_{1}  \\
-(1+ x_{5})a_{1} & (1+ x_{5})a_{2}       & 0                    & 0
\end{array}
\right). 
\end{align*}
\end{lem}


\subsection{Action of ${\rm SU}(2) \subset {\rm SO}(4) \times \{ 1 \}$} \label{section SU2}

The next lemma follows easily from the proof of Lemma \ref{orbit U2}.

\begin{lem}[Orbits of the ${\rm SU}(2)$-action] \label{orbit SU2}

By the ${\rm SU}(2)$-action, 
any point in $\Lambda^{2}_{-} S^{4}$ is mapped to a point 
in the fiber of 
$\underline{p_{0}} = {}^t\! (x_{1}, 0, 0, 0, x_{5})$ 
with $x_{1} \geq 0$.
The ${\rm SU}(2)$-orbit through 
$p_{0} \in \Lambda^{2}_{-} S^{4}|_{\underline{p_{0}}}$ 
is diffeomorphic to 
\begin{align*}
\left\{
\begin{array}{ll}
S^{3} & \qquad \mbox{for }\  x_{5} \neq \pm 1, \\
S^{2} & \qquad \mbox{for }\  x_{5} = 1, p_{0} \neq 0, \\
*       & \qquad \mbox{for }\  x_{5} = 1, p_{0} = 0 \mbox{ or }\ x_{5} = -1.
\end{array}
\right.
\end{align*}
\end{lem}

Define the basis $\{ E_{1}, E_{2}, E_{3} \}$ of the Lie algebra $\mathfrak{su}(2)$ of ${\rm SU}(2)$ by   
\begin{align} \label{basis of su2}
E_{1} = 
\left( 
\begin{array}{cc}
0  & 1 \\
-1 & 0 \\
\end{array}
\right), \qquad
E_{2} = 
\left( 
\begin{array}{cc}
0  & i  \\
i   & 0 \\
\end{array}
\right), \qquad
E_{3} = 
\left( 
\begin{array}{cc}
i  & 0   \\
0 & -i  \\
\end{array}
\right), \qquad
\end{align}
which satisfies 
$[E_{j}, E_{j+1}] = 2 E_{j+2}$ for  $j \in \mathbb{Z}/3$. 
Note that 
via the inclusion ${\rm SU}(2) \hookrightarrow {\rm SO}(4) \times \{ 1 \}$, 
$E_{1}, E_{2}$ and $E_{3}$ correspond to 
\begin{align} \label{lie alg of diag su2}
\left(
\begin{array}{ccc}
           &  I_{2}   & \\
    -I_{2} &         & \\
           &          &0
\end{array} 
\right), \qquad
\left(
\begin{array}{ccc}
       &  J & \\
    J &     & \\
       &     &0
\end{array} 
\right), \qquad
\left(
\begin{array}{ccc}
J &      &  \\
   & -J &   \\
   &     & 0
\end{array} 
\right), 
\end{align}
where $J = 
\left(
\begin{array}{cc}
 &   -1 \\
 1&     
\end{array}
\right),
$
respectively. 
Since $E_{i}$ in (\ref{basis of su2})
agrees with $E_{i}$ in (\ref{basis of u2}) for $i=1,2,3$, 
we have the same formula as Lemma \ref{data U2}.


\subsection{${\rm SO}(3) = {\rm SO}(3) \times \{ I_{2} \}$-action} \label{section SO3}

Use the notation in Section \ref{second local frame}. 

\begin{lem}[Orbits of the ${\rm SO}(3)$-action] \label{orbit SO3}

By the ${\rm SO}(3)$-action, 
any point in $\Lambda^{2}_{-} S^{4}$ is mapped to a point 
in the fiber of 
$\underline{p_{0}} = {}^t\! (x_{1}, 0, 0, x_{4}, x_{5})$ 
for some $x_{1} \geq 0$.
The ${\rm SO}(3)$-orbit through 
$p_{0} =  ({}^t\! (x_{1}, 0, 0, x_{4}, x_{5}), {}^t\! (a_{1}, a_{2}, a_{3})) 
\in \Lambda^{2}_{-} S^{4}|_{\underline{p_{0}}}$ 
is diffeomorphic to 
\begin{align*}
\left\{
\begin{array}{ll}
{\rm SO}(3) & \qquad \mbox{for }\  x_{1} > 0, (a_{2}, a_{3}) \neq 0, \\ 
 S^{2}       & \qquad \mbox{for }\  x_{1} > 0, (a_{2}, a_{3}) = 0 \mbox{ or }\ x_{1} = 0, p_{0} \neq 0, \\ 
*             &  \qquad \mbox{for }\  x_{1} = 0, p_{0} = 0.
\end{array}
\right.
\end{align*}
\end{lem}

\begin{proof}
We easily see the cases 
$x_{1} > 0$ and $x_{1} = 0, x_{5} \neq \pm 1$ from 
(\ref{SO3 times 1 action}). 
Suppose that $x_{1}=0, x_{5} = \pm 1$. 
Then the stabilizer of the ${\rm SO}(3)$-action on $S^{4}$ at 
$\underline{p_{0}} = {}^t\! (0, 0, 0, 0, \pm 1)$ 
is ${\rm SO}(3)$.  
By using the notation in Section \ref{frame at pole}, 
the action of $g = (g_{i j}) \in {\rm SO}(3)$ is given by 
\begin{align*}
(g_{*} f_{1}, g_{*} f_{2}, g_{*} f_{3}, g_{*} f_{4}) 
=
(f_{1}, f_{2}, f_{3}, f_{4}) 
\left( 
\begin{array}{cc}
g & \\
   & 1
\end{array} 
\right)
\end{align*}
at $\underline{p_{0}}$. 
The induced action of 
$
g =  (g_{i j}) \in {\rm SO}(3)
$
on $\Lambda^{2}_{-} S^{4}|_{\underline{p_{0}}}$ 
is described as 
\begin{align*}
\left( 
\begin{array}{c}
A_{1} \\
A_{2} \\
A_{3} \\
\end{array} 
\right)
\mapsto
\left( 
\begin{array}{ccc}
g_{33} & -g_{32}  & -g_{31} \\
-g_{23} & g_{22}  &  g_{21} \\
-g_{13} & g_{12}  & g_{11} \\
\end{array} 
\right)
\left( 
\begin{array}{c}
A_{1} \\
A_{2} \\
A_{3} \\
\end{array} 
\right), 
\end{align*}
which gives the proof in the case $x_{5} = \pm 1$. 
\end{proof}

By Lemma \ref{orbit SO3}, 
an ${\rm SO}(3)$-orbit 
through $({}^t\! (x_{1}, 0, 0, x_{4}, x_{5}), {}^t\! (a_{1}, a_{2}, a_{3}))$ 
is 3-dimensional when 
$x_{1} > 0, (a_{2}, a_{3}) \neq 0$. 
By the fact that the stabilizer at its point  
is ${\rm SO}(2)$
and (\ref{SO3 times 1 action}), 
its ${\rm SO}(3)$-orbit contains a point 
$({}^t\! (x_{1}, 0, 0, x_{4}, x_{5}), {}^t\! (a_{1}, a_{2}, 0))$, 
where $x_{1} > 0, a_{2} > 0$. 
Thus we may assume that 
$x_{1} > 0, a_{2} > 0, a_{3} = 0$.

Let $\{ E_{i} \}_{1 \leq i \leq 3}$ be the basis of $\mathfrak{so}(3)$ in (\ref{basis of so3}).
At $p_{0} = \left( {}^t\! (x_{1}, 0, 0, x_{4}, x_{5}), {}^t\! (a_{1}, a_{2}, 0) \right)$, 
the vector fields $\tilde{E_{i}^{*}}$ on $\Lambda^{2}_{-} S^{4}$ generated by $E_{i}$ 
are described as 
\begin{align*}
\tilde{E_{1}^{*}} &= 
\frac{x_{1}}{\sqrt{1-x_{5}^{2}}} 
(x_{1} e_{1} + x_{4} e_{2}) 
- a_{2} \frac{\partial}{\partial a_{1}} +  a_{1} \frac{\partial}{\partial a_{2}},\\
\tilde{E_{2}^{*}} &=   
\frac{x_{1}}{\sqrt{1-x_{5}^{2}}} 
(-x_{4} e_{1} + x_{1} e_{2}) - a_{1} \frac{\partial}{\partial a_{3}},\\
\tilde{E_{3}^{*}} &= - a_{2} \frac{\partial}{\partial a_{3}}, 
\end{align*}
by (\ref{SO3 times 1 action}). 
A straightforward computation gives the following.
\begin{lem} \label{data SO3}
At $p_{0} = \left( {}^t\! (x_{1}, 0, 0, x_{4}, x_{5}), {}^t\! (a_{1}, a_{2}, 0) \right)$, 
we have 
\begin{align*}
\left( 
\pi^{*} \omega_{j} (\tilde{E_{1}^{*}}, \tilde{E_{2}^{*}}) 
\right)
&= (x_{1}^{2}, 0, 0), \\
\left( 
\pi^{*} \omega_{j} (\tilde{E_{1}^{*}}, \tilde{E_{3}^{*}}) 
\right)
&= 0, \\
\left( 
\pi^{*} \omega_{j} (\tilde{E_{2}^{*}}, \tilde{E_{3}^{*}}) 
\right)
&= 0, 
\end{align*}
\begin{align*}
\left( b_{i} (\tilde{E_{j}^{*}}) \right) 
&=
\left(
\begin{array}{ccc}
\left(-1+ \frac{x_{1}^{2}}{1-x_{5}} \right) a_{2}  & - \frac{x_{1} x_{4}}{1-x_{5}} a_{2}          & 0 \\
\left(1 - \frac{x_{1}^{2}}{1-x_{5}} \right) a_{1}  & \frac{x_{1} x_{4}}{1-x_{5}} a_{1}             & 0 \\
\frac{x_{1} x_{4}}{1-x_{5}} a_{1}                     & \left(-1+ \frac{x_{1}^{2}}{1-x_{5}} \right) a_{1} & -a_{2} 
\end{array}
\right).
\end{align*}
\end{lem}

\subsection{Irreducible ${\rm SO}(3)$-action} \label{section irr SO3}

The irreducible representation of ${\rm SO}(3)$ on $\mathbb{R}^{5}$
is described as follows. 

Let $V$ be the space of all $3 \times 3$ real symmetric traceless matrices, 
which is isomorphic to $\mathbb{R}^{5}$. 
Let ${\rm SO}(3)$ act on $V$ by 
$g \cdot X = g X g^{-1}$, where $X \in V, g \in {\rm SO}(3)$. 
This action preserves the norm $|X|^{2} = {\rm tr} (X^{2})/2$, 
and hence induces the action on the unit sphere 
$S^{4} = \{ X \in V;  |X| = 1 \} \subset V$. 
We identify $V \cong \mathbb{R}^{5}$ by 
\begin{align} \label{identification V R5}
\left(
\begin{array}{ccc}
\lambda_{1} & \mu_{1}      & \mu_{2} \\
\mu_{1}      & \lambda_{2} & \mu_{3} \\
\mu_{2}      & \mu_{3}      & - \lambda_{1}-\lambda_{2} 
\end{array}
\right)
\mapsto
{}^t\! \left(\lambda_{1} + \frac{\lambda_{2}}{2}, -\mu_{2}, \mu_{3}, \mu_{1}, - \frac{\sqrt{3}}{2} \lambda_{2} \right). 
\end{align}

\begin{rem}
We can also describe 
the irreducible representation of ${\rm SO}(3) = {\rm SU}(2)/ \mathbb{Z}_{2}$ on $\mathbb{R}^{5}$ 
by the method in Appendix \ref{real irr rep}. 
We use the description above 
because it is easier to work with. 
\end{rem}

Use the notation in Section \ref{second local frame}. 

\begin{lem}[Orbits of the irreducible ${\rm SO}(3)$-action] \label{orbit irr SO3}
By the ${\rm SO}(3)$-action, 
any point in $\Lambda^{2}_{-} S^{4}$ is mapped to a point 
in the fiber of 
$\underline{p_{0}} = {}^t\! (x_{1}, 0, 0, 0, x_{5})$ 
where $x_{1} > 0, |x_{5}| \leq 1/2$.

When $ |x_{5}| < 1/2$, 
the ${\rm SO}(3)$-orbit through 
$p_{0} =  ({}^t\! (x_{1}, 0, 0, 0, x_{5}), {}^t\! (a_{1}, a_{2}, a_{3})) 
\in \Lambda^{2}_{-} S^{4}|_{\underline{p_{0}}}$ 
is diffeomorphic to 
\begin{align*}
\left\{
\begin{array}{ll}
{\rm SO}(3) & 
\qquad \mbox{when }\  a_{1} a_{2} a_{3} \neq 0 \mbox{ or one of } \{ a_{1}, a_{2}, a_{3} \} \mbox{ is } 0, \\ 
{\rm SO}(3)/ \mathbb{Z}_{2} & \qquad \mbox{when }\  
\mbox{two of } \{ a_{1}, a_{2}, a_{3} \} \mbox{ are } 0, \\ 
{\rm SO}(3)/(\mathbb{Z}_{2} \times \mathbb{Z}_{2}) & 
\qquad \mbox{when }\  a_{1}=a_{2}=a_{3}= 0.  
\end{array}
\right.
\end{align*}

When $ x_{5} = 1/2$ (resp. $ x_{5} = -1/2$), 
the ${\rm SO}(3)$-orbit through 
$p_{0} =  ({}^t\! (x_{1}, 0, 0, 0, x_{5}),$ ${}^t\! (a_{1}, a_{2}, a_{3})) 
\in \Lambda^{2}_{-} S^{4}|_{\underline{p_{0}}}$ 
is diffeomorphic to 
\begin{align*}
\left\{
\begin{array}{ll}
{\rm SO}(3) & \qquad \mbox{for }\  a_{2} \neq 0, (a_{1}, a_{3}) \neq 0, \\ 
 S^{2}       & \qquad \mbox{for }\  a_{2} \neq 0, (a_{1}, a_{3}) = 0, \\ 
{\rm SO}(3)/\mathbb{Z}_{2} & \qquad \mbox{for }\  a_{2} = 0, (a_{1}, a_{3}) \neq 0, \\ 
\mathbb{R} P^{2} & \qquad \mbox{for }\  a_{1}=a_{2}=a_{3}= 0. 
\end{array}
\right. \\
\left (
\mbox{resp. }
\left\{
\begin{array}{ll}
{\rm SO}(3) & \qquad \mbox{for }\ a_{1} \neq 0, (a_{2}, a_{3}) \neq 0, \\ 
 S^{2}       & \qquad \mbox{for }\  a_{1} \neq 0, (a_{2}, a_{3}) = 0, \\ 
{\rm SO}(3)/\mathbb{Z}_{2} & \qquad \mbox{for }\  a_{1} = 0, (a_{2}, a_{3}) \neq 0, \\ 
\mathbb{R} P^{2} & \qquad \mbox{for }\   a_{1}=a_{2}=a_{3}= 0. 
\end{array}
\right.
\right)
\end{align*}
\end{lem}

\begin{rem}
The ${\rm SO}(3)$-orbit in $S^{4}$ 
through $ {}^t\! (\sqrt{3}, 0, 0, 0, \pm 1)/2$ 
is a superminimal surface called a Veronese surface. 
For example, see \cite{Hitchin}.
\end{rem}

\begin{proof}
The first statement is well-known. See, for example, \cite{Bor}. 
Set 
\begin{align*}
\Sigma = \left \{ 
{\rm diag} (\lambda_{1}, \lambda_{2}, \lambda_{3});
\lambda_{1} \geq \lambda_{2} \geq \lambda_{3}, 
\lambda_{1} + \lambda_{2} + \lambda_{3} = 0, 
\lambda_{1}^{2} + \lambda_{2}^{2} + \lambda_{3}^{2} = 2
\right \}. 
\end{align*}
Since every symmetric matrix is diagonalizable by an orthogonal matrix, unique up to the order
of its diagonal elements, we see that 
every orbit of the ${\rm SO}(3)$-action on $S^{4}$ intersects $\Sigma$ 
at precisely one point.
Via (\ref{identification V R5}), $\Sigma$ corresponds to 
\begin{align*}
\left \{
{}^t\! (x_{1}, 0, 0, 0, x_{5}) \in S^{4}; x_{1} > 0, -\frac{1}{2} \leq x_{5} \leq \frac{1}{2}
\right \}. 
\end{align*}

The stabilizer at $\underline{p_{0}} = {}^t\! (x_{1}, 0, 0, 0, x_{5})$, 
where  $|x_{5}| < 1/2$,  
is given by 
\begin{align*}
\left \{ 
{\rm diag} (\epsilon_{1}, \epsilon_{2}, \epsilon_{1} \epsilon_{2});
\epsilon_{1} = \pm 1, \epsilon_{2} = \pm 1
\right \}. 
\end{align*}
Note that 
\begin{align*}
e_{1}   &=  {}^t\! (0, 1, 0, 0, 0),  \qquad
&e_{2} &=  {}^t\! (0, 0, 1, 0, 0), \\
e_{3}   &=  {}^t\! (0, 0, 0, 1, 0), \qquad
&e_{4} &=  {}^t\! (-x_{5}, 0, 0, 0, x_{1})
\end{align*}
at $\underline{p_{0}}$. 
Then via the identification (\ref{identification V R5}), 
the action of 
$k = {\rm diag}(\epsilon_{1}, \epsilon_{2}, \epsilon_{1} \epsilon_{2})$ is given by 
\begin{align*}
(k_{*} e_{1}, k_{*} e_{2}, k_{*} e_{3}, k_{*} e_{4})
=
(\epsilon_{2} e_{1}, \epsilon_{1} e_{2}, \epsilon_{1} \epsilon_{2} e_{3}, e_{4}), 
\end{align*}
which induces the action of $k$ on $\Lambda^{2}_{-} S^{4}|_{\underline{p_{0}}}$ described as 
\begin{align*}
{}^t\! (a_{1}, a_{2}, a_{3})
\mapsto {}^t\! (\epsilon_{1} \epsilon_{2} a_{1}, \epsilon_{1} a_{2}, \epsilon_{2} a_{3}).
\end{align*}

The stabilizer at $\underline{p_{0}} = {}^t\! (\sqrt{3}/2, 0, 0, 0, \pm 1/2)$ 
is given by 
\begin{align*}
\left \{ 
\left(
\begin{array}{cc}
\det A & \\
          & A
\end{array}
\right);
A \in {\rm O}(2)
\right \}, \qquad
\left \{ 
\left(
\begin{array}{cc}
A & \\
   & \det A
\end{array}
\right);
A \in {\rm O}(2)
\right \},
\end{align*}
respectively. 
The induced action of 
$\left(
\begin{array}{cc}
\det A & \\
          & A
\end{array}
\right)$
(resp. 
$\left(
\begin{array}{cc}
A & \\
   & \det A
\end{array}
\right)$), 
where $A = (a_{i j}) \in {\rm O}(2)$, on $\Lambda^{2}_{-} S^{4}|_{\underline{p_{0}}}$ is given by 
\begin{align*}
\det A
\left(
\begin{array}{ccc}
a_{11} (a_{11}^{2} - 3 a_{12}^{2}) & 0 &  a_{12} (3 a_{11}^{2} - a_{12}^{2})\\
0                                        & 1 & 0 \\
a_{21} (3 a_{11}^{2} - a_{12}^{2}) & 0 & a_{22} (a_{11}^{2} - 3 a_{12}^{2})
\end{array}
\right) \qquad
\left( \mbox{resp. }
\det A
\left(
\begin{array}{ccc}
1 & 0 & 0 \\
0 & a_{22} &  a_{21} \\
0 & a_{12} & a_{11}
\end{array}
\right)
\right).
\end{align*}
Hence we obtain the statement.
\end{proof}

The irreducible representation of ${\rm SO}(3)$ on $\mathbb{R}^{5}$ gives 
the inclusion ${\rm SO}(3) \hookrightarrow {\rm SO}(5)$. 
Via this inclusion, 
the basis $\{ E_{1}, E_{2}, E_{3} \}$ of $\mathfrak{so}(3)$ in (\ref{basis of so3}) 
correspond to
\begin{align} \label{lie alg of irr su2}
\left(
\begin{array}{cccc}
           & J & \\
J  &         & 
\begin{array}{c}
0 \\
\sqrt{3}
\end{array}
\\
 & 
\begin{array}{cc}
0 & - \sqrt{3}
\end{array}
\end{array} 
\right), 
\left(
\begin{array}{ccc}
-2J &      &  \\
    & -J &   \\
    &     & 0
\end{array} 
\right), 
\left(
\begin{array}{cccc}
         & -I_{2} & \\
 I_{2}  &         & 
\begin{array}{c}
- \sqrt{3} \\
0
\end{array}
\\
 & 
\begin{array}{cc}
\sqrt{3} & 0
\end{array}
\end{array} 
\right), 
\end{align}
where $J = 
\left(
\begin{array}{cc}
 &   -1 \\
 1&     
\end{array}
\right),
$
respectively. 
Let $E_{i}^{*}$ be the vector field on $S^{4}$ generated by $E_{i}$. 
Then we have 
at $\underline{p_{0}} = {}^t\! (x_{1}, 0, 0, 0, x_{5}) \in S^{4} $, 
where $x_{1} > 0, |x_{5}| \leq 1/2$, 
\begin{align*}
([E_{i}^{*}, e_{j}]) 
=& 
\frac{\sqrt{3}}{x_{1}}
\left( 
\begin{array}{cccc}
-x_{5} e_{2}   & x_{5} e_{1}             & e_{4}                    & -e_{3} \\
0               & \sqrt{3} x_{1} e_{3}  & -\sqrt{3} x_{1} e_{2} & 0\\
-x_{5} e_{3} & -e_{4} &  x_{5} e_{1}      & e_{2}
\end{array}
\right), \\
(L_{E_{i}^{*}} \omega_{j}) = &
\left( 
\begin{array}{ccc}
0                 & \frac{\sqrt{3} (1+x_{5})}{x_{1}} \omega_{3} & -\frac{\sqrt{3} (1+x_{5})}{x_{1}} \omega_{2} \\
3 \omega_{2}  &  -3 \omega_{1}                                  &0 \\
-\frac{\sqrt{3} (1+x_{5})}{x_{1}} \omega_{3}  & 0 & \frac{\sqrt{3} (1+x_{5})}{x_{1}} \omega_{1}
\end{array}
\right). 
\end{align*}
Hence 
at $p_{0} = \left( {}^t\! (x_{1}, 0, 0, 0, x_{5}), {}^t\! (a_{1}, a_{2}, a_{3}) \right) \in \Lambda^{2}_{-} S^{4}$, 
where $x_{1} > 0, |x_{5}| \leq 1/2$, 
the vector fields $\tilde{E_{i}^{*}}$ on $\Lambda^{2}_{-} S^{4}$ generated by $E_{i}$ 
are described as 
\begin{align*}
\tilde{E_{1}^{*}} &= 
(x_{1} + \sqrt{3} x_{5}) e_{3} + \frac{\sqrt{3} (1+x_{5})}{x_{1}} 
\left( a_{3} \frac{\partial}{\partial a_{2}} - a_{2} \frac{\partial}{\partial a_{3}} \right), \\
\tilde{E_{2}^{*}} &= 
-2 x_{1} e_{1} + 
3 \left( a_{2} \frac{\partial}{\partial a_{1}} - a_{1} \frac{\partial}{\partial a_{2}} \right), \\
\tilde{E_{3}^{*}} &=
(x_{1} - \sqrt{3} x_{5}) e_{2} + \frac{\sqrt{3} (1+x_{5})}{x_{1}} 
\left( - a_{3} \frac{\partial}{\partial a_{1}} + a_{1} \frac{\partial}{\partial a_{3}} \right). 
\end{align*}
A straightforward computation gives the following.
\begin{lem} \label{data irrSO3}
At $p_{0} = \left( {}^t\! (x_{1}, 0, 0, 0, x_{5}), {}^t\! (a_{1}, a_{2}, a_{3}) \right)$, we have 
\begin{align*}
\left( 
\pi^{*} \omega_{j} (\tilde{E_{1}^{*}}, \tilde{E_{2}^{*}}) 
\right)
&= (0, 2 x_{1} (x_{1} + \sqrt{3} x_{5}), 0), \\
\left( 
\pi^{*} \omega_{j} (\tilde{E_{1}^{*}}, \tilde{E_{3}^{*}}) 
\right)
&= (0, 0, x_{1}^{2}-3 x_{5}^{2}), \\
\left( 
\pi^{*} \omega_{j} (\tilde{E_{2}^{*}}, \tilde{E_{3}^{*}}) 
\right)
&= (2 x_{1} (- x_{1} + \sqrt{3} x_{5}), 0, 0), 
\end{align*}
\begin{align*}
\left( b_{i} (\tilde{E_{j}^{*}}) \right) 
=
\left(
\begin{array}{ccc}
0                                         & (1-2 x_{5}) a_{2}    & -(\sqrt{3} x_{1} + x_{5} + 1) a_{3} \\
(\sqrt{3} x_{1} - x_{5} - 1) a_{3}  &  (-1+2 x_{5}) a_{1}  & 0 \\
(-\sqrt{3} x_{1} + x_{5} + 1) a_{2} & 0                       & (\sqrt{3} x_{1} + x_{5} + 1) a_{1}
\end{array}
\right).
\end{align*}
\end{lem}


\subsection{Classification of homogeneous  coassociative submanifolds}

Summarizing the results in Section \ref{section orbits of Lie subgrp}, 
we obtain the following. 
\begin{prop} \label{Lie grp 3 4-dim orbit}
The connected closed Lie subgroup of ${\rm SO}(5)$ which has a 4-dimensional orbit 
on $\Lambda^{2}_{-} S^{4}$
is either ${\rm SO}(5)$, whose only 4-dimensional orbit is the zero section, 
${\rm SO}(3) \times {\rm SO}(2)$, or ${\rm U}(2)$. 

The connected closed Lie subgroup of ${\rm SO}(5)$ which has a 3-dimensional orbit 
on $\Lambda^{2}_{-} S^{4}$ 
is one of the following. 
\begin{align*}
{\rm SO}(4) &= {\rm SO}(4) \times \{ 1 \},  \qquad
&&{\rm SO}(3) \times {\rm SO}(2), \qquad
{\rm U}(2), \ {\rm SU}(2) \subset {\rm SO}(4) \times \{ 1 \},  \\
{\rm SO}(3) &= {\rm SO}(3) \times \{ I_{2} \}, \qquad
&&{\rm SO}(3) \mbox{ acting irreducibly on } \mathbb{R}^{5}.
\end{align*}
\end{prop}

By Proposition \ref{Lie grp 3 4-dim orbit}, 
we prove Theorem \ref{homog coasso}.


\begin{proof}[Proof of Theorem \ref{homog coasso}]
By Proposition \ref{Lie grp 3 4-dim orbit}, 
we consider the actions of 
${\rm SO}(5), {\rm SO}(3) \times {\rm SO}(2)$, and ${\rm U}(2)$. 
A 4-dimensional ${\rm SO}(5)$-orbit, which is the zero section, is obviously coassociative. 

Consider the ${\rm SO}(3) \times {\rm SO}(2)$-action. 
Use the notation in Section \ref{section SO3 SO2}.
By Lemma \ref{orbit SO3 SO2}, 
the ${\rm SO}(3) \times {\rm SO}(2)$ 
orbits through 
$p_{0}=\left({}^t\! (x_{1}, 0, 0, x_{4}, 0),  {}^t\! (a_{1}, a_{2}, a_{3}) \right)$, 
where 
$0 < x_{1} < 1, (a_{2}, a_{3}) \neq 0$, 
are 4-dimensional.
By (\ref{def of G2 str S4}) and Lemma \ref{data SO3SO2}, we compute 
\begin{align*}
\varphi_{\lambda} (\tilde{E_{1}^{*}}, \tilde{E_{2}^{*}}, \tilde{E_{3}^{*}}) &= 
\varphi_{\lambda} (\tilde{E_{1}^{*}}, \tilde{E_{2}^{*}}, \tilde{E_{4}^{*}}) = 0,\\
\varphi_{\lambda} (\tilde{E_{1}^{*}}, \tilde{E_{3}^{*}}, \tilde{E_{4}^{*}}) &= 
2 s_{\lambda} x_{1} x_{4} (a_{2} x_{1} -a_{3} x_{4}), \\
\varphi_{\lambda} (\tilde{E_{2}^{*}}, \tilde{E_{3}^{*}}, \tilde{E_{4}^{*}}) &= 
-2 s_{\lambda} x_{1} x_{4} (a_{2} x_{4} + a_{3} x_{1}), 
\end{align*}
at $p_{0}$.
Hence the orbit is coassociative if and only if 
$x_{1} =0$ or $x_{4}=0$ or $a_{2}=a_{3}=0$, 
which implies that the orbit is not 4-dimensional. 

Consider the ${\rm U}(2)$-action. 
Use the notation in Section \ref{section U2}. 
By Lemma \ref{orbit U2}, 
the ${\rm U}(2)$ 
orbits through 
$p_{0}=\left({}^t\! (x_{1}, 0, 0, 0, x_{5}),  {}^t\! (a_{1}, a_{2}, a_{3}) \right)$, 
where 
$x_{5} \neq \pm 1, (a_{1}, a_{2}) \neq 0$, 
are 4-dimensional.
By (\ref{def of G2 str S4}) and Lemma \ref{data U2}, we compute 
\begin{align*}
\varphi_{\lambda} (\tilde{E_{1}^{*}}, \tilde{E_{2}^{*}}, \tilde{E_{3}^{*}}) &= 
\varphi_{\lambda} (\tilde{E_{1}^{*}}, \tilde{E_{2}^{*}}, \tilde{E_{4}^{*}}) = 0,\\
\varphi_{\lambda} (\tilde{E_{1}^{*}}, \tilde{E_{3}^{*}}, \tilde{E_{4}^{*}}) &= 
- 4 s_{\lambda} x_{1}^{2} a_{2}, \\
\varphi_{\lambda} (\tilde{E_{2}^{*}}, \tilde{E_{3}^{*}}, \tilde{E_{4}^{*}}) &= 
- 4 s_{\lambda} x_{1}^{2}  a_{1}, 
\end{align*}
at $p_{0}$.
Hence the orbit is coassociative if and only if 
$x_{1} =0$ or $a_{1}=a_{2}=0$, 
which implies that the orbit is not 4-dimensional. 
\end{proof}


\section{Cohomogeneity one coassociative submanifolds} \label{cohomo one section}

The connected Lie subgroups which have 3-dimensional orbits are classified 
in Proposition \ref{Lie grp 3 4-dim orbit}. 
We construct cohomogeneity one coassociative submanifolds
in each case. 
In this section, denote by $I \subset \mathbb{R}$ an open interval. 

\subsection{${\rm SO}(4) = {\rm SO}(4) \times \{ 1 \}$-action} \label{cohomo1 section SO4}

By Lemma \ref{orbit SO4}, an ${\rm SO}(4)$-orbit 
through $({}^t\! (x_{1}, 0, 0, 0, x_{5}), {}^t\! (0, 0, 0))$, where $x_{1}>0$, 
is 3-dimensional. 
We may find a path 
$c: I \rightarrow \Lambda^{2}_{-} S^{4}$ given by 
\begin{align*}
c(t) = \left( {}^t\! (x_{1}(t), 0, 0, 0, x_{5}(t)), {}^t\! (0, 0, 0) \right) 
\end{align*}
satisfying $x_{1}(t) > 0, \varphi_{\lambda} |_{{\rm SO}(4) \cdot {\rm Image}(c)} = 0$. 
However, since 
${\rm SO}(4) \cdot {\rm Image}(c)$ is contained in the zero section 
which is an obvious coassociative submanifold, 
we cannot find new examples in this case. 


\subsection{${\rm SO}(3) \times {\rm SO}(2)$-action} \label{cohomo1 section SO3SO2}

We give a proof of Theorem \ref{SO3SO2 thm}. 
Recall the notation in Section \ref{section SO3 SO2}. 
By Lemma \ref{orbit SO3 SO2}, 
an ${\rm SO}(3) \times {\rm SO}(2)$-orbit 
through $({}^t\! (x_{1}, 0, 0, x_{4}, 0), {}^t\! (a_{1}, a_{2}, a_{3}))$ 
is 3-dimensional when 
\begin{enumerate}
\item 
$0 < x_{1} < 1, (a_{2}, a_{3}) = 0$,
\item
$x_{1} = 1, (a_{2}, a_{3}) \neq 0$, or
\item
$x_{1} = 0, (a_{1}, a_{2}, a_{3}) \neq 0$. 
\end{enumerate}

Consider \underline{case 1}. 
Take a path $c: I \rightarrow \Lambda^{2}_{-} S^{4}$ given by 
\begin{align*}
c(t) = \left( {}^t\! (x_{1}(t), 0, 0, x_{4}(t), 0), {}^t\! (a_{1}(t), 0, 0) \right), 
\end{align*}
where $x_{1}(t), x_{4}(t) > 0$. 
Note that $\tilde{E_{3}^{*}}=0$ at $c(t)$. 
We find a path $c$ satisfying $\varphi_{\lambda} |_{({\rm SO}(3) \times {\rm SO}(2)) \cdot {\rm Image}(c)} = 0$, 
where $\varphi_{\lambda}$ is given by (\ref{def of G2 str S4}).  
We easily see that $\varphi_{\lambda}(\tilde{E_{i}^{*}}, \tilde{E_{j}^{*}}, \tilde{E_{k}^{*}})|_{c} = 0$
for $1 \leq i,j,k \leq 4$ by Lemma \ref{data SO3SO2}. 
Since 
$\dot{c} = (-\dot{x}_{1} x_{4} + x_{1} \dot{x}_{4}) e_{3} + \dot{a}_{1} \frac{\partial}{\partial a_{1}}$ 
and 
\begin{align*}
(\pi^{*} \omega_{i} (\tilde{E_{j}^{*}}, \dot{c})) 
= 
\left( 
 \begin{array}{cccc}
0                    & 0                        & 0  & -\dot{x}_{1} \\
x_{1} \dot{x}_{4}  & -x_{4} \dot{x}_{4}    & 0  & 0\\
-x_{4} \dot{x}_{4}  & -x_{1} \dot{x}_{4}  & 0 & 0\\
\end{array}
\right), \qquad 
(b_{i} (\dot{c})) 
= 
\left( 
 \begin{array}{c}
\dot{a}_{1}\\
0\\
0
\end{array}
\right), 
\end{align*}
we have at $c(t)$
\begin{align*}
\varphi_{\lambda} (\tilde{E_{1}^{*}}, \tilde{E_{2}^{*}}, \dot{c}) 
&= 2 s_{\lambda} \left(-2 a_{1} x_{4} \dot{x}_{4} + \dot{a}_{1} x_{1}^{2}
\right)
- s_{\lambda}^{-3} \dot{a}_{1} a_{1}^{2} x_{4}^{2}, \\
\varphi_{\lambda} (\tilde{E_{1}^{*}}, \tilde{E_{2}^{*}}, \tilde{E_{4}^{*}}) &= 
\varphi_{\lambda} (\tilde{E_{1}^{*}}, \tilde{E_{4}^{*}}, \dot{c}) = 
\varphi_{\lambda} (\tilde{E_{2}^{*}}, \tilde{E_{4}^{*}}, \dot{c}) =0.
\end{align*}
Thus 
the condition $\varphi_{\lambda} |_{({\rm SO}(3) \times {\rm SO}(2)) \cdot {\rm Image}(c)} = 0$ is equivalent to 
\begin{align*}
4 a_{1} x_{1} \dot{x}_{1} +
\frac{1}{\lambda+a_{1}^{2}}
\left \{ 
- a_{1}^{2}  
+ (2 \lambda + 3 a_{1}^{2}) x_{1}^{2}
\right \}
\dot{a}_{1} = 0.
\end{align*}
This equation is solved explicitly as 
\begin{align} \label{sol SO3SO2}
G(a_{1}, x_{1}) = C
\end{align}
for  $C \in \mathbb{R}$, 
where 
$G: \mathbb{R} \times (0, 1] \rightarrow \mathbb{R}$ 
is defined by 
\begin{align} \label{def of G}
G(a_{1}, x_{1}) = 
a_{1} (\lambda + a_{1}^{2})^{1/4} (2 x_{1}^{2} - 1) + 
\frac{1}{2} \int^{a_{1}}_{0} \frac{x^{2}+2 \lambda}{(\lambda + x^{2})^{3/4}} dx. 
\end{align}
This solution is obtained by Maple 16 \cite{Maple}.

\begin{rem} \label{domain of G}
We give some remarks on the domain of $G$. 
Since we take a path 
$
c(t) = \left( {}^t\! (x_{1}(t), 0, 0, x_{4}(t), 0), {}^t\! (a_{1}(t), 0, 0) \right) 
$
satisfying $0 < x_{1}(t) < 1$, 
$G$ is defined on $\mathbb{R} \times (0,1)$ in the first place. 
Though $G$ extends to a map $\mathbb{R} \times [0,1] \rightarrow \mathbb{R}$ formally, 
it is not appropriate to define $G$ on $x_{1}=0$. 

In fact, by (\ref{SO3 times 1 action}) and (\ref{1 times SO2 action}), 
the ${\rm SO}(3) \times {\rm SO}(2)$-orbit through 
$\left( {}^t\! (0, 0, 0, 1, 0), {}^t\! (a_{1}, 0, 0) \right)$ 
coincides with that through $\left( {}^t\! (0, 0, 0, 1, 0), {}^t\! (-a_{1}, 0, 0) \right)$. 
Thus we should have $G(- a_{1}, 0) = G(a_{1}, 0)$. 
However, we easily see that $G(- a_{1}, 0) = -G(a_{1}, 0)$. 

Such a problem does not occur when $x_{1}=1$. 
Hence we regard $G$ as a map $\mathbb{R} \times (0,1] \rightarrow \mathbb{R}$.
\end{rem}

Set 
\begin{align*}
M_{C} &= 
{\rm SO}(3) \times {\rm SO}(2) \cdot 
\left \{ 
\left(
{}^t\! (x_{1}, 0, 0, \sqrt{1-x_{1}^{2}}, 0), {}^t\! (a_{1}, 0, 0)
\right); 
\begin{array}{c}
G(a_{1} ,x_{1}) = C, \\
a_{1} \in \mathbb{R}, 0 < x_{1} \leq 1
\end{array}
\right \}, \\
M_{C}^{\pm} &= 
{\rm SO}(3) \times {\rm SO}(2) \cdot 
\left \{ 
\left(
{}^t\! (x_{1}, 0, 0, \sqrt{1-x_{1}^{2}}, 0), {}^t\! (a_{1}, 0, 0)
\right); 
\begin{array}{c}
G(a_{1} ,x_{1}) = C, \\
\pm a_{1} >0, 0 < x_{1} \leq 1
\end{array}
\right \}. 
\end{align*}
Then 
$M_{C}$ is coassociative  
and $M_{C} = M_{C}^{+} \sqcup M_{C}^{-}$ when $C \neq 0$ 
and $M_{0} = M_{0}^{+} \sqcup M_{0}^{-} \sqcup S^{4}$.

\begin{lem}
The coassociative submanifold $M_{C}$ is homeomorphic to 
\begin{align*}
\left \{
\begin{array}{ll}
(S^{2} \times \mathbb{R}^{2}) \sqcup (S^{2} \times S^{1} \times \mathbb{R}_{> 0}) 
& \qquad \mbox{for }\ C \neq 0, \\
S^{4} \sqcup (S^{2} \times S^{1} \times \mathbb{R}_{> 0}) \sqcup (S^{2} \times S^{1} \times \mathbb{R}_{> 0}) 
& \qquad \mbox{for }\ C=0,
\end{array}
\right.
\end{align*}
where $S^{4}$ is the zero section of $\Lambda^{2}_{-} S^{4}$. 
\end{lem}

\begin{proof}
Since we have 
\begin{align*}
\frac{\partial G}{\partial x_{1}} = 4 a_{1} (\lambda + a_{1}^{2})^{1/4} x_{1}, \qquad
\frac{\partial G}{\partial a_{1}} = 2^{-1} (\lambda + a_{1}^{2})^{-3/4} 
\left \{ (2 x_{1}^{2} - 1) (3 a_{1}^{2} + 2\lambda)  + a_{1}^{2} + 2 \lambda \right \}, 
\end{align*}
$G(a_{1}, \cdot)$ is monotonically increasing (resp. decreasing) 
on $( 0, 1]$ for a fixed $a_{1} > 0$ (resp. $a_{1} <0$) 
and 
$\lim_{x_{1} \to 0} G(\cdot, x_{1})$ (resp. $G(\cdot, 1)$) is monotonically decreasing  
(resp. increasing) on $\mathbb{R}$. 
We compute 
\begin{align*}
G (0, \cdot) = 0, \qquad
\lim_{a_{1} \to \pm \infty} G (a_{1}, 1) = \pm \infty, \qquad
\lim_{a_{1} \to \pm \infty} \lim_{x_{1} \to 0} G (a_{1}, x_{1}) = \mp \infty, 
\end{align*}
where we use the estimate 
\begin{align*}
&\int^{a_{1}}_{0} \frac{x^{2}+2 \lambda}{(\lambda + x^{2})^{3/4}} dx 
\leq
a_{1}  \left \{ (\lambda + a_{1}^{2})^{1/4} +\lambda^{1/4} \right \} 
\qquad \mbox{ for } a_{1} \geq 0. 
\end{align*}

Thus for any $C \in \mathbb{R}$, 
there exists a unique $\alpha_{C} \in \mathbb{R}$ (resp.  $\beta_{C} \in \mathbb{R}$)
such that $C = G(\alpha_{C}, 1)$ (resp.  $C=\lim_{x_{1} \to 0} G(\beta_{C}, x_{1})$). 
Note that $C$ and $\alpha_{C}$ (resp. $\beta_{C}$) have the same (resp. opposite) sign.
Now, define a function $g_{C}: \mathbb{R} - \{ 0 \} \rightarrow \mathbb{R}$ by 
\begin{align*}
g_{C} (a_{1}) = 
a_{1}^{-1} (\lambda + a_{1}^{2})^{-1/4}
\left ( C- \frac{1}{2} \int^{a_{1}}_{0} \frac{x^{2}+2 \lambda}{(\lambda + x^{2})^{3/4}} dx \right).
\end{align*}
Note that $G(a_{1}, x_{1}) = C$ is equivalent to $2 x_{1}^{2} - 1 = g_{C} (a_{1})$. 
We may find the condition on $a_{1}$ so that $-1 < g_{C} (a_{1}) \leq 1$.

First, suppose that $\underline{C > 0}$. 
\begin{lem} \label{SO3SO2 condition on gc}
When $a_{1}>0$, $g_{C} (a_{1}) > -1$ holds 
and $g_{C} (a_{1}) \leq 1$ is equivalent to $a_{1} \geq \alpha_{C}$. 
When $a_{1}<0$, $g_{C} (a_{1}) < 1$ holds 
and $g_{C} (a_{1}) > -1$ is equivalent to $a_{1} < \beta_{C}$. 
\end{lem}

\begin{proof}
Suppose that $a_{1} > 0$. 
Then $g_{C}(a_{1}) > -1$ 
is equivalent to $C> \lim_{x_{1} \to 0} G(a_{1}, x_{1})$, which holds for any $a_{1}>0$
since $ \lim_{x_{1} \to 0} G(a_{1}, x_{1})<0$. 
The condition that $g_{C}(a_{1}) \leq 1$
is equivalent to $C = G(\alpha_{C}, 1) \leq G(a_{1}, 1)$. 
Since $G(\cdot, 1)$ is monotonically increasing, this is equivalent to $a_{1} \geq \alpha_{C}$. 
We can prove similarly when $a_{1} < 0$.
\end{proof}

Set 
$
\Gamma(C)^{\pm} = \{ (a_{1}, x_{1}) \in \mathbb{R} \times (0, 1] ; G(a_{1}, x_{1})=C, \pm a_{1} > 0 \}.
$
By Lemma \ref{SO3SO2 condition on gc}, 
we have homeomorphisms
$[\alpha_{C}, \infty) \cong \Gamma(C)^{+}$ and 
$(-\infty, \beta_{C}) \cong \Gamma(C)^{-}$ 
via $a_{1} \mapsto (a_{1}, \sqrt{(g_{C} (a_{1}) +1)/2}$.  
Then from (\ref{SO3 times 1 action}) and (\ref{1 times SO2 action}), it follows that 
\begin{align*}
M_{C}^{+}
= 
\left \{
\left(
\left( 
\begin{array}{c}
g_{11} x_{1} \\
g_{21} x_{1} \\
g_{31} x_{1} \\
\sqrt{1 - x_{1}^{2}} \cos \alpha  \\
\sqrt{1 - x_{1}^{2}} \sin \alpha 
\end{array}
\right), 
\left( 
\begin{array}{c}
g_{11} a_{1}\\
g_{21} a_{1}\\
-g_{31} a_{1}\\
\end{array}
\right)
\right);
\begin{array}{c}
0< x_{1} = \sqrt{(g_{C}(a_{1})+1)/2} \leq 1 \\
a_{1} \in [ \alpha_{C}, \infty),\\
(g_{ij}) \in {\rm SO}(3), \\
\alpha \in \mathbb{R}
\end{array}
\right \}, 
\end{align*}
which implies that 
$M_{C}^{+}$ is homeomorphic to 
$S^{2} \times (S^{1} \times [\alpha_{C}, \infty)/ (S^{1} \times \{ \alpha_{C} \})) 
\cong S^{2} \times \mathbb{R}^{2}$. 
In the same way, 
we see that $M_{C}^{-}$ is homeomorphic to $S^{2} \times S^{1} \times \mathbb{R}_{>0}$. 
We can prove the case $\underline{C<0}$ similarly.

When $\underline{C=0}$, 
we see that 
$|g_{C} (a_{1}) | < 1$ holds for any $a_{1} \neq 0$
as Lemma \ref{SO3SO2 condition on gc}. 
Then we have homeomorphisms
$(0, \infty) \cong \Gamma(0)^{+}$ and 
$(-\infty, 0) \cong \Gamma(0)^{-}$,  
and 
by Lemma \ref{orbit SO3 SO2}, 
we obtain $M_{0}^{\pm} \cong S^{2} \times S^{1} \times \mathbb{R}_{>0}$. 
\end{proof}

\begin{rem}
When $\lambda = 0$, the equation (\ref{sol SO3SO2}) is given by 
\begin{align} \label{Asymp_SO3SO2}
a_{1} |a_{1}|^{\frac{1}{2}} \left( 2 x_{1}^{2} - \frac{2}{3} \right) = C.
\end{align}
We exhibit the graph of (\ref{Asymp_SO3SO2}). 
The solid curve indicates the case $C>0$,  
the dashed curve indicates the case $C = 0$ 
and 
the dotted curve indicates the case $C<0$.   
We see that the solution (\ref{sol SO3SO2}) is  
asymptotic to this graph 
as $\lambda \rightarrow 0$.
The vertical line gives a coassociative cone in 
$\Lambda^{2}_{-} S^{4} - \{ \mbox{zero section} \} \cong 
\mathbb{C} P^{3} \times \mathbb{R}_{>0}$, 
which corresponds to a Lagrangian submanifold 
in the nearly K\"{a}hler $\mathbb{C}P^{3}$. 
\end{rem}
\begin{figure}
 \begin{center}
  \includegraphics[width=5cm]{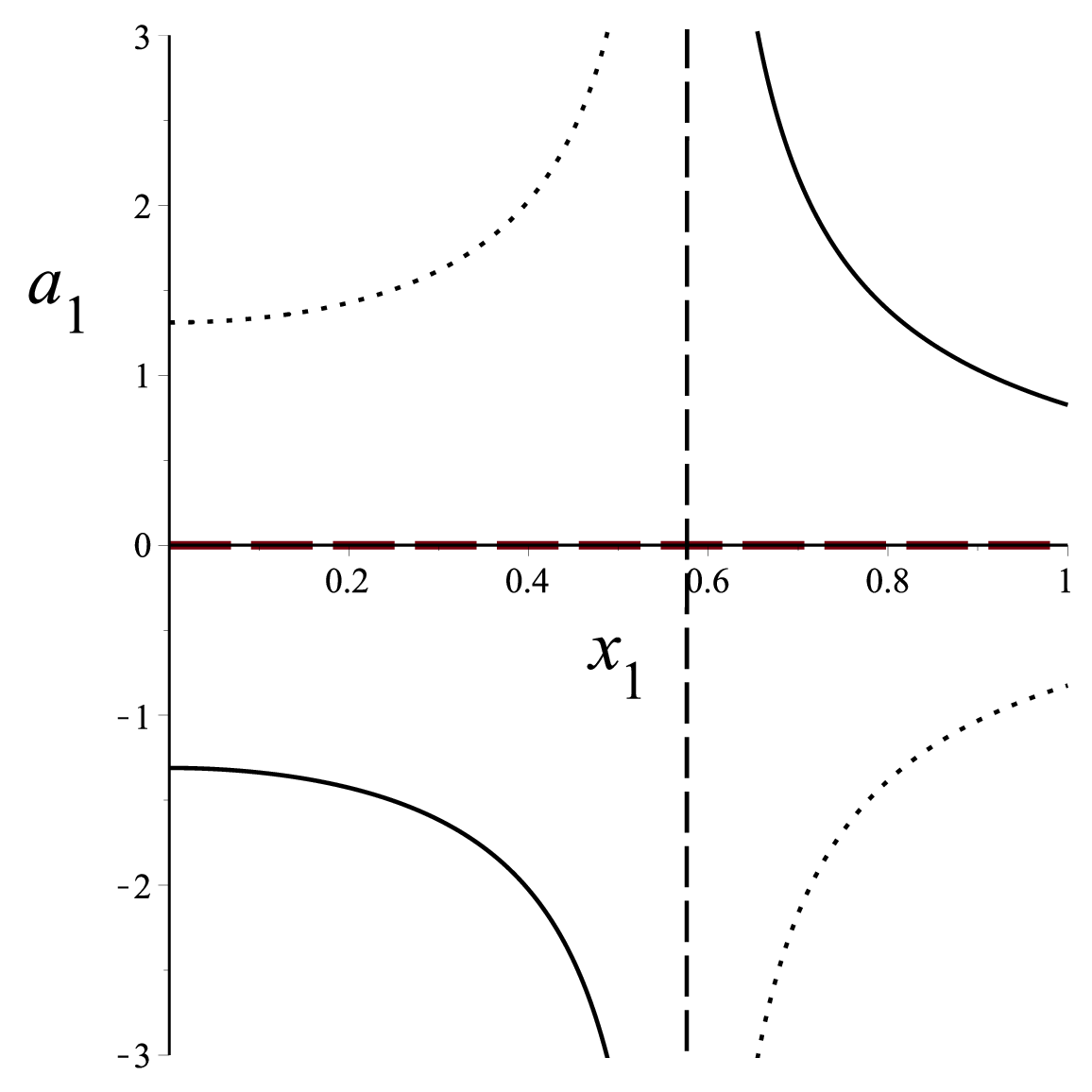}
  \caption{\hspace{-0.2cm} the graph of (\ref{Asymp_SO3SO2})}
 \end{center}
\end{figure}

Consider \underline{case 2}.
Take a path $c: I \rightarrow \Lambda^{2}_{-} S^{4}$ given by 
\begin{align*}
c(t) = \left( {}^t\! (1, 0, 0, 0, 0), {}^t\! (a_{1}(t), a_{2} (t), a_{3} (t)) \right). 
\end{align*}
We may assume that $a_{2}>0, a_{3}=0$ so that 
$c(t)$ is transverse to the ${\rm SO}(3) \times {\rm SO}(2)$-orbit. 
We find a path $c$ satisfying $\varphi_{\lambda} 
|_{({\rm SO}(3) \times {\rm SO}(2)) \cdot {\rm Image} (c)} = 0$, 
where $\varphi_{\lambda}$ is given by (\ref{def of G2 str S4}).  
Since 
$\dot{c} = \sum_{i=1}^{2} \dot{a}_{i} \frac{\partial}{\partial a_{i}}$, 
we have at $c(t)$
\begin{align*}
(\pi^{*} \omega_{i} (\tilde{E_{j}^{*}}, \dot{c})) 
&= 0, \\
(b_{i} (\dot{c})) 
&= 
\left( 
\dot{a}_{1}, \dot{a}_{2}, 0
\right), \\
\varphi_{\lambda} (\tilde{E_{1}^{*}}, \tilde{E_{2}^{*}}, \dot{c}) 
&= 2 s_{\lambda} \dot{a}_{1}, \\
\varphi_{\lambda} (\tilde{E_{p}^{*}}, \tilde{E_{q}^{*}}, \dot{c}) &= 0
\qquad \mbox{for }\ (p,q) \neq (1,2),(2,1). 
\end{align*}
Thus 
the condition $\varphi_{\lambda} |_{({\rm SO}(3) \times {\rm SO}(2)) \cdot {\rm Image} (c)} = 0$ is equivalent to 
$a_{1}=C$ for $C \in \mathbb{R}$. 
Set
$M_{C} = {\rm SO}(3) \times {\rm SO}(2) \cdot \{ ( {}^t\! (1, 0, 0, 0, 0), {}^t\! (C, r, 0)); r \in \mathbb{R} \}$.
By 
(\ref{SO3 times 1 action}) and 
(\ref{1 times SO2 action}), 
$M_{C}$ is explicitly described as 
\begin{align} \label{SO3SO2 case2 MC}
M_{C} =& \left \{ 
\left(
\begin{array}{c}
g_{11} \\
g_{21} \\
g_{31} \\
0 \\
0
\end{array}
\right),
\left( 
\begin{array}{ccc}
g_{11}  & g_{12}   & -g_{13} \\
g_{21}  & g_{22}   & -g_{23} \\
-g_{31} & -g_{32} & g_{33} \\
\end{array}
\right)
\left( 
\begin{array}{ccc}
1  & 0  & 0 \\
0  & u  & -v \\
0  & v  & u \\
\end{array}
\right)
\left( 
\begin{array}{c}
C \\
r \\
0 \\
\end{array}
\right); 
\begin{array}{c}
(g_{ij}) \in {\rm SO}(3) \\
u^{2} + v^{2} =1 \\
r \in \mathbb{R}
\end{array} 
\right \} \\
= &
\left \{ 
\left( {}^t\! (x_{1}, x_{2}, x_{3}, 0, 0), {}^t\! (a_{1}, a_{2}, a_{3}) \right) \in S^{4} \times \mathbb{R}^{3}; 
a_{1} x_{1} + a_{2} x_{2} - a_{3} x_{3} = C 
\right \}. \nonumber 
\end{align}
Thus 
$M_{C}$ is canonically identified with 
$\{ (v, w) \in S^{2} \times \mathbb{R}^{3}; \langle v, w \rangle = C \}$, 
which is homeomorphic to 
$\{ (v, w) \in S^{2} \times \mathbb{R}^{3}; \langle v, w \rangle = 0 \} = TS^{2}$ 
via $(v, w) \mapsto (v, w-Cv)$.

\begin{rem} \label{relation Kari_Min}
Let $\Sigma^{2} \subset S^{4}$ be an oriented 2-submanifold. 
Let $L \rightarrow \Sigma$ be a line bundle over $\Sigma$ 
spanned by ${\rm vol}_{\Sigma} - * {\rm vol}_{\Sigma}$, 
where ${\rm vol}_{\Sigma}$ is a volume form of $\Sigma$ and $*$ is 
a Hodge star in $S^{4}$. 
Denote by $L^{\perp}$ the orthogonal complement bundle of $L$
in $\Lambda^{2}_{-} S^{4}$ and 
take a section $\eta$ of $L$ over $\Sigma$.
By the argument in \cite{Kari_Min} and \cite{Kari_Leung}, 
\begin{align*}
\eta + L^{\perp} = \left \{ (x, \eta_{x} + \sigma) \in \Lambda^{2}_{-} S^{4}|_{\Sigma}; 
x\in \Sigma, \sigma \in L^{\perp}_{x} \right \}
\end{align*}
is coassociative if and only if 
$\Sigma$ is superminimal and $\eta \in \mathbb{R} ({\rm vol}_{\Sigma} - * {\rm vol}_{\Sigma})$.

The submanifold $M_{C}$ is a special case of these examples.
In fact, 
$\pi (M_{C})$ is a totally geodesic $S^{2} = \{ {}^t\!(x_{1}, x_{2}, x_{3}, 0, 0) \in S^{4}\}$
and define 
$\tau \in C^{\infty}(S^{2}, \Lambda^{2}_{-} S^{4}|_{S^{2}})$
by 
\begin{align*}
\tau_{x} = x_{1} (\omega_{1})_{x} + x_{2} (\omega_{2})_{x} - x_{3} (\omega_{3})_{x},
\end{align*} 
where $x= {}^t\!(x_{1}, x_{2}, x_{3}, 0, 0) \in S^{2}$. 
Note that $\tau$ is not the restriction of the tautological 2-form to $S^{2}$. 
We easily see that 
$M_{C} = C \tau + (\mathbb{R} \tau)^{\perp}$ and 
$\tau = {\rm vol}_{S^{2}} - * {\rm vol}_{S^{2}}$ by the ${\rm SO}(3)$-invariance of $\tau$.
\end{rem}

Consider \underline{case 3}.
Take a path $c: I \rightarrow \Lambda^{2}_{-} S^{4}$ given by 
\begin{align*}
c(t) = \left( {}^t\! (0, 0, 0, 1, 0), {}^t\! (a_{1}(t), a_{2} (t), a_{3} (t)) \right). 
\end{align*}
We may assume that $a_{1} >0, a_{2}=a_{3}=0$ so that 
$c(t)$ is transverse to the ${\rm SO}(3) \times {\rm SO}(2)$-orbit. 
Since $\tilde{E_{1}^{*}} = a_{1} \frac{\partial}{\partial a_{2}}$, 
$\tilde{E_{2}^{*}} = - a_{1} \frac{\partial}{\partial a_{3}}$, 
$\tilde{E_{3}^{*}} = 0, \tilde{E_{4}^{*}} = e_{4}$ at $c(t)$, 
we compute 
$
\varphi_{\lambda} (\tilde{E_{1}^{*}}, \tilde{E_{2}^{*}}, \dot{c}) 
= - a_{1}^{2} \dot{a}_{1}/s_{\lambda}^{3}, 
$
which implies that $a_{1}$ is constant. 
Hence we cannot obtain a 4-submanifold.


\subsection{Action of ${\rm U}(2) \subset {\rm SO}(4) \times \{ 1 \}$} \label{cohomo1 section U2}

By Lemma \ref{orbit U2}, a ${\rm U}(2)$-orbit 
through $p_{0} = ({}^t\! (x_{1}, 0, 0, 0, x_{5}), {}^t\! (0, 0, a_{3}))$, where $x_{1}>0$, 
is 3-dimensional. 
At $p_{0}$, the stabilizer of the ${\rm U}(2)$-action is ${\rm U}(1)$. 
Thus a ${\rm U}(2)$-orbit through $p_{0}$ 
agrees with an ${\rm SU}(2)$-orbit through $p_{0}$. 
The case of ${\rm SU}(2)$ is considered in the next subsection.


\subsection{Action of ${\rm SU}(2) \subset {\rm SO}(4) \times \{ 1 \}$} \label{cohomo1 section SU2}

We give a proof of Theorem \ref{SU2 thm}. 
Recall the notation in Section \ref{section U2} and \ref{section SU2}. 
By Lemma \ref{orbit SU2}, 
an ${\rm SU}(2)$-orbit 
through $({}^t\! (x_{1}, 0, 0, 0, x_{5}), {}^t\! (a_{1}, a_{2}, a_{3}))$ 
is 3-dimensional when 
$x_{5} \neq 0$. 
Take a path 
$c: I \rightarrow \Lambda^{2}_{-} S^{4}$ given by 
\begin{align*}
c(t) = \left( {}^t\! (x_{1}(t), 0, 0, 0, x_{5}(t)), {}^t\! (a_{1}(t), a_{2}(t), a_{3}(t)) \right), 
\end{align*}
where $x_{1}(t) > 0$. 
We find a path 
$c$ satisfying $\varphi_{\lambda}|_{{\rm SU}(2) \cdot {\rm Image}(c)}=0$, 
where $\varphi_{\lambda}$ is given by (\ref{def of G2 str S4}).  
The condition $\varphi_{\lambda} (\tilde{E_{1}^{*}}, \tilde{E_{2}^{*}}, \tilde{E_{3}^{*}})=0$ 
is always satisfied. 
In fact, 
since the $G_{2}$-structure $\varphi_{\lambda}$ is preserved by the ${\rm SU}(2)$-action, we have 
$d(\varphi_{\lambda} (\tilde{E_{1}^{*}}, \tilde{E_{2}^{*}}, \tilde{E_{3}^{*}})) = 0$ by Cartan's formula. 
Since the action of ${\rm SU}(2)$ is not free, 
we have $\tilde{E_{1}^{*}} \wedge \tilde{E_{2}^{*}} \wedge \tilde{E_{3}^{*}} = 0$ 
at some point. Thus we have 
$\varphi_{\lambda} (\tilde{E_{1}^{*}}, \tilde{E_{2}^{*}}, \tilde{E_{3}^{*}}) = 0$. 

\begin{lem} 
The condition $\varphi_{\lambda}|_{{\rm SU}(2) \cdot {\rm Image}(c)} = 0$ is equivalent to 
\begin{align} \label{S4SU2_3}
4 \dot{a}_{i} \frac{1-x_{5}}{1+x_{5}}
- a_{i} \frac{d}{dt} \left \{ \log(\lambda + r^{2}) + 8 \log (1 + x_{5}) \right \} = 0  
\qquad \mbox{ for } i = 1, 2, 3.
\end{align}
\end{lem}

\begin{proof}
Since 
$
\dot{c}(t) =
\left( -\dot{x}_{1} x_{5} + x_{1} \dot{x}_{5} \right) e_{4} 
+ \sum_{j=1}^{3} \dot{a}_{j} \frac{\partial}{\partial a_{j}}, 
$
we have 
\begin{align*}
\left (
\pi^{*} \omega_{i} (\tilde{E_{j}^{*}}, \dot{c})
\right ) &= 
\left( 
\begin{array}{ccc}
 0             & -\dot{x}_{5} & 0 \\
-\dot{x}_{5} & 0              & 0 \\
             0 & 0              & \dot{x}_{5} \\
\end{array}
\right),  \\
b_{j} (\dot{c}) &= \dot{a_{j}} \qquad \mbox{ for } j=1,2,3. 
\end{align*}
Then we compute 
\begin{align*}
\varphi_{\lambda} (\tilde{E_{1}^{*}}, \tilde{E_{2}^{*}}, \dot{c}) |_{c}
=& \ 
2 s_{\lambda}
\sum_{i = 1}^{3} 
b_{i} \wedge \pi^{*} \omega_{i} (\tilde{E_{1}^{*}}, \tilde{E_{2}^{*}}, \dot{c})
+ \frac{1}{s_{\lambda}^{3}} b_{123}(\tilde{E_{1}^{*}}, \tilde{E_{2}^{*}}, \dot{c}) \\
=& 
2 s_{\lambda} \left( -2 (1+x_{5}) \dot{x}_{5} a_{3} +\dot{a}_{3} x_{1}^{2} \right) 
-\frac{(1+x_{5})^{2} a_{3} }{2 s_{\lambda}^{3}} \frac{d(r^{2})}{dt} \\
=&
\frac{s_{\lambda} (1+x_{5})^{2}}{2}
\left \{
4 \dot{a}_{3} \frac{1 - x_{5}}{1 + x_{5}}
-a_{3} 
\frac{d}{dt}
\left (
\log(\lambda + r^{2}) + 8 \log (1+x_{5})
\right )
\right \}. 
\end{align*}
We compute $\varphi_{\lambda} (E_{i}^{*}, E_{i+1}^{*}, \dot{c}) |_{c}$ 
in the same way and we see the lemma. 
\end{proof}

By (\ref{S4SU2_3}), 
we have 
$\frac{d}{dt} {}^t\!
\left (
a_{1}(t), a_{2}(t), a_{3}(t)
\right )
=
f(t)
{}^t\!
\left (
a_{1}(t), a_{2}(t), a_{3}(t)
\right )$
for some function $f(t)$. The solution is given by 
${}^t\! \left ( a_{1}(t), a_{2}(t), a_{3}(t) \right ) = 
\exp(\int^{t} f(s)ds) v$ for some $v \in \mathbb{R}^{3}$. 
Thus we may assume that 
$
{}^t\! \left (a_{1}(t), a_{2}(t), a_{3}(t) \right )
=
r(t) v
$
for a smooth function $r : I \rightarrow \mathbb{R}_{\geq 0}$ 
and $v \in S^{2} \subset \mathbb{R}^{3}$. 
Then (\ref{S4SU2_3}) is solved explicitly as 
\begin{align} \label{S4SU2_5}
F(r, x_{5}) = C
\end{align}
for $C \in \mathbb{R}$, 
where $F: [0, \infty) \times [-1, 1] \rightarrow \mathbb{R}$ is defined by 
\begin{align} \label{def of F}
F(r, x_{5}) = 
(1 - 3 x_{5}) (\lambda + r^{2})^{1/8} \sqrt{r}
+ \int^{\sqrt{r}}_{0} \frac{2 \lambda}{(\lambda + x^{4})^{7/8}} dx. 
\end{align}
This solution is obtained by Maple 16 \cite{Maple}. 
Though the definition of $c$ implies that 
the domain of $F$ is $[0, \infty) \times (-1, 1)$, 
$F$ extends to a map $[0, \infty) \times [-1, 1] \rightarrow \mathbb{R}$
as in Remark \ref{domain of G}. 
Thus we obtain the coassociative submanifold 
\begin{align*}
M_{C, v} := {\rm SU}(2) \cdot \left \{ 
\left({}^t\! (\sqrt{1-x_{5}^{2}}, 0, 0, 0, x_{5}), r v \right) ; 
                                 F(r, x_{5}) = C, r\geq 0, -1 \leq x_{5} \leq 1 \right \}, 
\end{align*}
where $C \in \mathbb{R}$ and $v \in S^{2} \subset \mathbb{R}^{3}$. 
We study the topology of $M_{C, v}$ now.

\begin{lem} \label{topology_S4SU2}
The coassociative submanifold $M_{C}$ is homeomorphic to 
\begin{align*}
\left \{
\begin{array}{ll}
\mathbb{R}^{4} & \qquad \mbox{for }\ C>0, \\
S^{4} \sqcup (S^{3} \times \mathbb{R}_{>0}) & \qquad \mbox{for }\ C=0, \\
\mathcal{O}_{\mathbb{C} P^{1}} (-1) & \qquad \mbox{for }\  C<0,
\end{array}
\right.
\end{align*}
where $S^{4}$ is the zero section of $\Lambda^{2}_{-} S^{4}$
and $\mathcal{O}_{\mathbb{C} P^{1}} (-1)$ is 
the tautological line bundle over $\mathbb{C} P^{1} \cong S^{2}$. 
\end{lem}

\begin{proof}
Since we have 
\begin{align*}
\frac{\partial F}{\partial x_{5}} &= -3 (\lambda + r^{2})^{1/8} \sqrt{r}, \\
\frac{\partial F}{\partial r} &= 
4^{-1} r^{-1/2} (\lambda + r^{2})^{-7/8} 
\left \{ (1 - 3 x_{5})(3 r^{2} + 2\lambda)   +  4 \lambda \right \}, 
\end{align*}
$F(r, \cdot)$ is monotonically decreasing 
on $[ -1, 1]$ for a fixed $r > 0$
and 
$F(\cdot, -1)$ (resp. $F(\cdot, 1)$) is monotonically increasing 
(resp. decreasing) on $\mathbb{R}_{\geq 0}$. 
We compute 
\begin{align*}
F (0, \cdot) = 0, \qquad
\lim_{r \to \infty} F (r, \mp 1) = \pm \infty. 
\end{align*}
Thus for any $C > 0$ (resp. $C<0$), 
there exists a unique $\alpha_{C} >0$ (resp.  $\beta_{C} >0$)
such that $C = F(\alpha_{C}, -1)$ (resp.  $C=F(\beta_{C}, 1)$). 

Now, define a function $f_{C}: \mathbb{R}_{>0} \rightarrow \mathbb{R}$ by 
\begin{align*}
f_{C} (r) = 
r^{-1/2} (\lambda + r^{2})^{-1/8}
\left ( C- \int^{\sqrt{r}}_{0} \frac{2 \lambda}{(\lambda + x^{4})^{7/8}} dx \right).
\end{align*}
Note that $F(x_{5}, t) = C$ is equivalent to $1-3 x_{5} = f_{C} (r)$. 
Since $-1 \leq x_{5} \leq 1$, 
we may find the condition on $a_{1}$ so that $-2 \leq f_{C} (r) \leq 4$.

\begin{lem} \label{SU2 condition on fc}
When $C>0$, $f_{C} (r) > -2$ holds for any $r>0$ 
and $f_{C} (r) \leq 4$ is equivalent to $r \geq \alpha_{C}$. 
When $C<0$, $f_{C} (r) < 4$ holds for any $r>0$ 
and $f_{C} (r) \geq -2$ is equivalent to $r \geq \beta_{C}$. 
When $C=0$, $-2 < f_{C} (r) < 4$ holds for any $r>0$. 
\end{lem}

\begin{proof}
Suppose that $C > 0$. 
Then $f_{C}(r) > -2$ 
is equivalent to $C> F(r, 1)$, which holds for any $r>0$
since $ F(r, 1)<0$. 
The condition that $f_{C}(r) \leq 4$
is equivalent to $C = F(\alpha_{C}, -1) \leq F(r, -1)$. 
Since $F(\cdot, -1)$ is monotonically increasing, this is equivalent to $r \geq \alpha_{C}$. 
We can prove similarly when $C \leq 0$.
\end{proof}

\begin{rem} \label{graph cha SU2}
Set 
$
\Gamma(C) = \{ (x_{5},r) \in [-1, 1] \times [0, \infty) ; F(x_{5},r)=C \}.
$
By Lemma \ref{SU2 condition on fc}, 
we have homeomorphisms
$[\alpha_{C}, \infty) \cong \Gamma(C)$ when $C>0$, 
$[\beta_{C}, \infty) \cong \Gamma(C)$ when $C<0$, 
and 
$(0, \infty) \cong \Gamma(0) \cap \{ r \neq 0 \}$ 
via $r \mapsto ((1- f_{C}(r))/3, r)$.  
Note that 
$\Gamma(C) \cap \{ x_{5} = 1 \} = \emptyset$
when $C>0$, 
$\Gamma(C) \cap \{ x_{5} = -1 \} = \emptyset$ 
when $C<0$, 
and
$\Gamma(0)  \cap \{ r \neq 0 \} \cap \{ x_{5} = \pm 1 \} = \emptyset$. 
\end{rem}

Hence we see that 
\begin{align*}
M_{0, v} \cap \{ r > 0 \}
=
\left \{
\left(
\left(
\begin{array}{c}
\sqrt{1-x_{5}^{2}} a\\
\sqrt{1-x_{5}^{2}} b\\
x_{5}
\end{array}
\right),
r v
\right)
\in \mathbb{C}^{2} \oplus \mathbb{R} \oplus \mathbb{R}^{3}; 
\begin{array}{c}
-1 < x_{5} = \frac{1-f_{C}(r)}{3}<1, \\
r>0,\\
a,b \in \mathbb{C}, \\
|a|^{2} + |b|^{2} =1
\end{array}
\right \},
\end{align*}
which is homeomorphic to $S^{3} \times \mathbb{R}_{>0}$.

When $C \neq 0$, $M_{C, v}$ intersects with $\Lambda^{2}_{-} S^{4} |_{{}^t\! (0,0,0,0, \pm 1)}$. 
To study the topology of $M_{C, v}$, 
we use the stereographic local coordinates.

First, suppose that $\underline{C > 0}$. 
By Remark \ref{graph cha SU2}, 
$M_{C, v}$ does not intersect with $\Lambda^{2}_{-} S^{4} |_{{}^t\! (0,0,0,0,1)}$.
Take the stereographic local coordinates of $\Phi : S^{4} - \{ x_{5} = 1\} \rightarrow \mathbb{R}^{4}$ 
given by 
\begin{align*} 
\Phi (x_{1}, \cdots, x_{5}) &= 
\frac{\left( x_{1}, x_{2}, x_{3}, -x_{4} \right )}{1 - x_{5}}, \\
\Phi^{-1} (y_{1}, \cdots, y_{4}) &= \frac{\left( 2 y_{1}, 2 y_{2}, 2 y_{3}, -2 y_{4}, -1+|y|^{2} \right)}{1+|y|^{2}}, 
\end{align*}
where 
$|y|^{2} = \Sigma_{i = 1}^{4} y_{i}^{2}$. 
The standard metric on $S^{4}$ is given by  
$
4 \sum_{j=1}^{4} dy_{j}^{2}/(1 + |y|^{2})^{2}, 
$ and hence we see that 
$
\left \{ 2 dy_{i} /(1 + |y|^{2}) \right \}_{i = 1, \cdots ,4} 
$
is a local oriented orthonormal coframe. 
The trivialization 
$4 (1 + |y|^{2})^{-2} \{ dy_{12}-dy_{34}, dy_{13}-dy_{42}, dy_{14} - dy_{23} \}$ 
of $\Lambda^{2}_{-}S^{4}$
induces the local fiber coordinates $(\alpha_{1}, \alpha_{2}, \alpha_{3})$. 
Setting $\zeta_{1} = y_{1} + i y_{2}, \zeta_{2} = y_{3} + i y_{4}$, 
the action of ${\rm SU}(2)$ is described as 
\begin{align*}
\left( 
\begin{array}{cc}
a & - \overline{b} \\
b & \overline{a}
\end{array}
\right) \cdot 
( {}^t\! (\zeta_{1}, \zeta_{2}), {}^t\! (\alpha_{1}, \alpha_{2}, \alpha_{3}))
=
( {}^t\! (a \zeta_{1} - \overline{b} \overline{\zeta_{2}}, 
\overline{\beta} \overline{\zeta_{1}} + \alpha \zeta_{2}), 
{}^t\! (\alpha_{1}, \alpha_{2}, \alpha_{3})), 
\end{align*}
where $a, b \in \mathbb{C}$ such that $|a|^{2} + |b|^{2} = 1$. 
Then we obtain 
\begin{align*}
M_{C, v} 
=
\left \{
\left(
{}^t\! \left(
y_{1} a, y_{1} b \right),
r v'
\right)
\in \mathbb{C}^{2} \oplus \mathbb{R}^{3}; 
\begin{array}{c}
r \in [\alpha_{C}, \infty), 
y_{1} = \sqrt{\frac{6}{f_{C}(r)+2}-1}, \\
a,b \in \mathbb{C}, 
|a|^{2} + |b|^{2} =1
\end{array}
\right \},
\end{align*}
where 
$v' \in S^{2}$ is a corresponding element to $v$ under the change of local coordinates. 
Then it follows that 
$M_{C, v}$ is homeomorphic to 
$(S^{3} \times [\alpha_{C}, \infty))/(S^{3} \times \{ \alpha_{C} \}) \cong \mathbb{R}^{4}$.

Next, suppose that $\underline{C < 0}$. 
By Remark \ref{graph cha SU2}, 
$M_{C, v}$ does not intersect with $\Lambda^{2}_{-} S^{4} |_{{}^t\! (0,0,0,0,-1)}$.
Take the stereographic local coordinates of $\Psi : S^{4} - \{ x_{5} = -1 \} \rightarrow \mathbb{R}^{4}$ 
given by 
\begin{align*} 
\Psi (x_{1}, \cdots, x_{5}) &= 
\frac{\left( x_{1}, x_{2}, x_{3}, x_{4} \right )}{1 + x_{5}}, \\
\Psi^{-1} (u_{1}, \cdots, u_{4}) &= \frac{\left( 2 u_{1}, 2 u_{2}, 2 u_{3}, 2 u_{4}, 1-|u|^{2} \right)}{1+|u|^{2}}, 
\end{align*}
where 
$|u|^{2} = \Sigma_{i = 1}^{4} u_{i}^{2}$. 
The standard metric on $S^{4}$ is given by  
$
4 \sum_{j=1}^{4} du_{j}^{2}/(1 + |u|^{2})^{2}, 
$ and hence we see that 
$
\left \{ 2 du_{i} /(1 + |u|^{2}) \right \}_{i = 1, \cdots ,4} 
$
is a local oriented orthonormal coframe. 
The trivialization 
$4 (1 + |u|^{2})^{-2} \{ du_{12}-du_{34}, du_{13}-du_{42}, du_{14} - du_{23} \}$ 
of $\Lambda^{2}_{-}S^{4}$
induces the local fiber coordinates 
$(\underline{\alpha_{1}}, \underline{\alpha_{2}}, \underline{\alpha_{3}})$. 
Setting $\underline{\zeta_{1}} = u_{1} + i u_{2}, \underline{\zeta_{2}} = u_{3} + i u_{4}$, 
the action of ${\rm SU}(2)$ is described as  
\begin{align} \label{SU2 action stereo graphic coor}
g \cdot 
( {}^t\! (\underline{\zeta_{1}}, \underline{\zeta_{2}}), 
{}^t\! (\underline{\alpha_{1}}, \underline{\alpha_{2}}, \underline{\alpha_{3}}))
=
( g {}^t\! (\underline{\zeta_{1}}, \underline{\zeta_{2}}), 
\varpi (g) {}^t\! (\underline{\alpha_{1}}, \underline{\alpha_{2}}, \underline{\alpha_{3}})), 
\end{align}
where $g \in {\rm SU}(2)$ and $\varpi : {\rm SU}(2) \rightarrow {\rm SO}(3)$ is a double covering 
given by (\ref{covering_SU2_SO3}). 
Then we obtain 
\begin{align} \label{MC C<0 SU2}
M_{C, v} 
=
\left \{
\left(
g  {}^t\! (u_{1}, 0 \right),
r \varpi (g) v')
\in \mathbb{C}^{2} \oplus \mathbb{R}^{3}; 
\begin{array}{c}
r \in [\beta_{C}, \infty), 
u_{1} = \sqrt{\frac{6}{4- f_{C}(r)}-1}, \\
g \in {\rm SU}(2)
\end{array}
\right \},
\end{align}
where 
$v' \in S^{2}$ is a corresponding element to $v$ under the change of local coordinates. 

Note that the topology of $M_{C, v}$ is independent of $v$. 
In fact, fix $v_{0} \in S^{2}$ 
and let $v_{0}'$ be a corresponding element to $v_{0}$ under the change of local coordinates. 

For any $v \in S^{2}$, there exists $g_{0} \in {\rm SU}(2)$ such that $v' = \varpi (g_{0}) v_{0}'$. 
Then $M_{C, v} \cong M_{C, v_{0}}$ via 
$
(g  {}^t\! (u_{1}, 0),
r \varpi (g) v') 
\mapsto
(g g_{0}  {}^t\! (u_{1}, 0),
r \varpi (g) v'). 
$
Thus we only have to consider the case $v_{0}'= {}^t\! (1, 0, 0)$. 
Setting $v_{0}'= {}^t\! (1, 0, 0)$ in (\ref{MC C<0 SU2}), we obtain 
\begin{align*}
\left \{
\left(
{}^t\! (u_{1} a, u_{1} b \right),
r {}^t\! (|a|^{2}-|b|^{2}, 2 {\rm Im} (a \overline{b}), 2 {\rm Re} (a \overline{b})); 
\begin{array}{c}
r \in [\beta_{C}, \infty), 
u_{1} = \sqrt{\frac{6}{4- f_{C}(r)}-1}, \\
a,b \in \mathbb{C}, |a|^{2} + |b|^{2} = 1
\end{array}
\right \},
\end{align*}
which is homeomorphic to 
\begin{align*}
\left\{ (v, [w,r]) \in S^{2} \times (S^{3} \times [\beta_{C}, \infty))/(S^{3} \times \{ \beta_{C} \}); 
w \in p^{-1}(v) \right \},
\end{align*}
where $p: S^{3} \rightarrow \mathbb{C}P^{1} = S^{2}$ is the Hopf fibration. 
This is the tautological line bundle $\mathcal{O}_{\mathbb{C}P^{1}}(-1)$ over $\mathbb{C}P^{1}$. 
\end{proof}

\begin{rem}
When $\lambda = 0$, (\ref{S4SU2_5}) is given by  
\begin{eqnarray} \label{Asymp_SU(2)_S4}
(1- 3 x_{5}) r^{3/4}  = C. 
\end{eqnarray}
We exhibit the graph of (\ref{Asymp_SU(2)_S4}). 
The solid curve indicates the case $C>0$,  
the dashed curve indicates the case $C = 0$ 
and 
the dotted curve indicates the case $C<0$.   
We see that the solution (\ref{S4SU2_5}) is  
asymptotic to this graph 
as $\lambda \rightarrow 0$.
The vertical line gives a coassociative cone in 
$\Lambda^{2}_{-} S^{4} - \{\mbox{zero section} \} \cong 
\mathbb{C} P^{3} \times \mathbb{R}_{>0}$, 
which corresponds to a Lagrangian submanifold 
in the nearly K\"{a}hler $\mathbb{C}P^{3}$. 
\end{rem}

\begin{figure}
 \begin{center}
  \includegraphics[width=5cm]{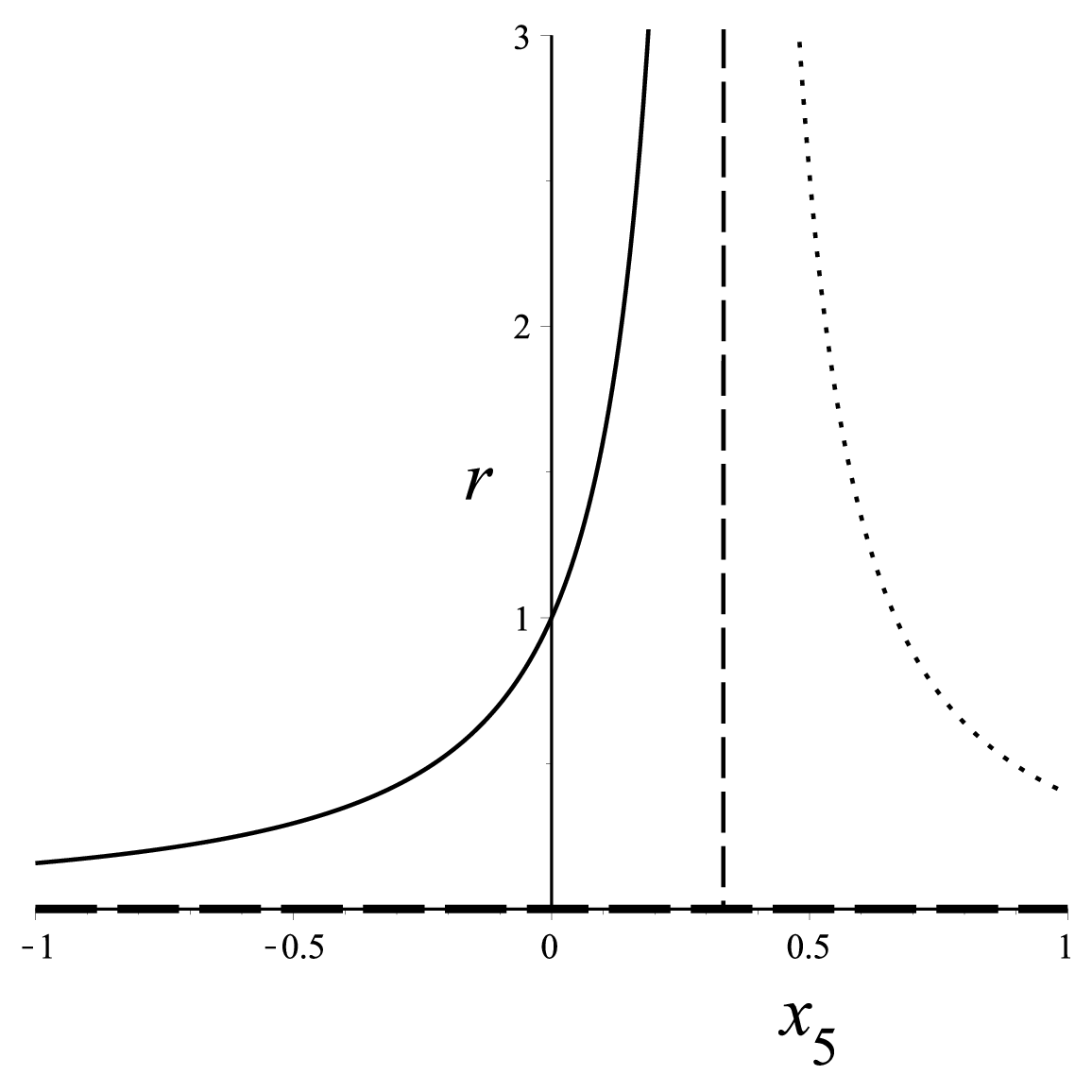}
  \caption{\hspace{-0.2cm} the graph of (\ref{Asymp_SU(2)_S4})}
 \end{center}
\end{figure}
%


\subsection{${\rm SO}(3) = {\rm SO}(3) \times \{ I_{2} \}$-action} \label{cohomo1 section SO3}

We give a proof of Theorem \ref{SO3 thm}.
Recall the notation in Section \ref{section SO3}. 
By Lemma \ref{orbit SO3}, 
an ${\rm SO}(3)$-orbit 
through $({}^t\! (x_{1}, 0, 0, x_{4}, x_{5}), {}^t\! (a_{1}, a_{2}, a_{3}))$ 
is 3-dimensional when 
$x_{1} > 0, (a_{2}, a_{3}) \neq 0$. 
Take a path $c: I \rightarrow \Lambda^{2}_{-} S^{4}$ given by 
\begin{align*}
c(t) = \left( {}^t\! (x_{1}(t), 0, 0, x_{4}(t), x_{5}(t)), {}^t\! (a_{1}(t), a_{2}(t), 0) \right), 
\end{align*}
where $x_{1}(t) > 0, a_{2}(t) > 0$. 
We assume that $a_{3} = 0$ so that $c(t)$ is transverse to the ${\rm SO}(3)$-orbits. 
We find a path $c$ satisfying $\varphi_{\lambda} |_{{\rm SO}(3) \cdot {\rm Image}(c)} = 0$, 
where $\varphi_{\lambda}$ is given by (\ref{def of G2 str S4}).  
We see that 
$\varphi_{\lambda} (\tilde{E_{1}^{*}}, \tilde{E_{2}^{*}}, \tilde{E_{3}^{*}})=0$ 
as in Section \ref{cohomo1 section SU2}.

\begin{lem}
The condition $\varphi_{\lambda} |_{{\rm SO}(3) \cdot {\rm Image}(c)} = 0$ 
is equivalent to
\begin{align}
4(2x_{1} \dot{x}_{1} a_{1} + \dot{a}_{1} x_{1}^{2}) - (1-x_{1}^{2}) a_{1} \frac{d}{dt} \log (\lambda + r^{2}) &= 0, 
\label{SO3 eq1}\\
4x_{1} \dot{x}_{1} 
- \frac{d}{dt} \log (\lambda + r^{2})
+
\frac{x_{1}^{2}}{1-x_{5}}
\left(
4 \dot{x}_{5}
+ 
\frac{d}{dt} \log (\lambda + r^{2})
\right)
&=0, \label{SO3 eq2}\\
4 \dot{x}_{4} 
+ \frac{x_{4}}{1-x_{5}}
\left(
4 \dot{x}_{5} 
+
\frac{d}{dt} \log (\lambda + r^{2})
\right)
&=0, \label{SO3 eq3}
\end{align}
where $r^{2} = a_{1}^{2} + a_{2}^{2}$.
\end{lem}

\begin{proof}

Since 
$
\dot{c} = 
\frac{-\dot{x}_{1} x_{4} + x_{1} \dot{x}_{4}}{\sqrt{1-x_{5}^{2}}} e_{3} 
+ \frac{\dot{x}_{5}}{\sqrt{1-x_{5}^{2}}} e_{4}
+ \dot{a}_{1} \frac{\partial}{\partial a_{1}} + \dot{a}_{2} \frac{\partial}{\partial a_{2}}, 
$
we have 
\begin{align*}
(\pi^{*} \omega_{i} (\tilde{E_{j}^{*}}, \dot{c})) 
= 
\left( 
 \begin{array}{ccc}
0                  & 0                     & 0  \\
x_{1} \dot{x}_{4} + \frac{x_{1} x_{4} \dot{x}_{5}}{1 - x_{5}} & x_{1} \dot{x}_{1} + \frac{x_{1}^{2} \dot{x}_{5}}{1 - x_{5}}   & 0  \\
x_{1} \dot{x}_{1} + \frac{x_{1}^{2} \dot{x}_{5}}{1 - x_{5}} & -x_{1} \dot{x}_{4} - \frac{x_{1} x_{4} \dot{x}_{5}}{1 - x_{5}}   & 0 
\end{array}
\right), 
(b_{i} (\dot{c})) 
= 
\left( 
 \begin{array}{c}
\dot{a}_{1}\\
\dot{a}_{2}\\
\frac{-\dot{x}_{1} x_{4} + x_{1} \dot{x}_{4}}{1-x_{5}} a_{2} 
\end{array}
\right).
\end{align*}
Then we compute 
\begin{align*}
\sum_{i=1}^{3} b_{i} \wedge \pi^{*} \omega_{i} (\tilde{E_{1}^{*}}, \tilde{E_{2}^{*}}, \dot{c})
=&
2 x_{1} a_{1} 
\left \{
- 
\frac{x_{1} x_{4} \dot{x}_{4}}{1-x_{5}}
+
\left(
1- \frac{x_{1}^{2}+x_{4}^{2}}{1-x_{5}}
\right)
\frac{x_{1} \dot{x}_{5}}{1-x_{5}} 
\right.
\\
&\left.
+\left( 
1- \frac{x_{1}^{2}}{1-x_{5}}
\right) \dot{x}_{1}
\right \}
+ \dot{a}_{1} x_{1}^{2}, \\
b_{123} (\tilde{E_{1}^{*}}, \tilde{E_{2}^{*}}, \dot{c}) 
=& - 
\left \{
\left( 
1 - \frac{x_{1}^{2}}{1-x_{5}}
\right)^{2}
+\frac{x_{1}^{2} x_{4}^{2}}{(1-x_{5})^{2}}
\right \} 
a_{1} (a_{1} \dot{a}_{1} + a_{2} \dot{a}_{2})
\end{align*}
Since 
$x_{1}^{2} + x_{4}^{2} + x_{5}^{2} =1, x_{1} \dot{x}_{1} + x_{4} \dot{x}_{4} + x_{5} \dot{x}_{5} = 0$, 
it follows that 
\begin{align*}
\sum_{i=1}^{3} b_{i} \wedge \pi^{*} \omega_{i} (\tilde{E_{1}^{*}}, \tilde{E_{2}^{*}}, \dot{c})
=
2 x_{1} \dot{x}_{1} a_{1} + \dot{a}_{1} x_{1}^{2}, \qquad
b_{123} (\tilde{E_{1}^{*}}, \tilde{E_{2}^{*}}, \dot{c}) 
= - \frac{(1-x_{1}^{2}) a_{1}}{2} \frac{d (r^{2})}{dt},
\end{align*}
which implies (\ref{SO3 eq1}).
In the same way, we compute
\begin{align*}
\sum_{i=1}^{3} b_{i} \wedge \pi^{*} \omega_{i} (\tilde{E_{1}^{*}}, \tilde{E_{3}^{*}}, \dot{c})
&=
a_{2} \left( x_{1} \dot{x}_{1} + \frac{x_{1}^{2} \dot{x}_{5}}{1 - x_{5}} \right), \\
b_{123} (\tilde{E_{1}^{*}}, \tilde{E_{3}^{*}}, \dot{c}) 
&= \left(-1 + \frac{x_{1}^{2}}{1-x_{5}} \right) \frac{a_{2}}{2} \frac{d (r^{2})}{dt},\\
\sum_{i=1}^{3} b_{i} \wedge \pi^{*} \omega_{i} (\tilde{E_{2}^{*}}, \tilde{E_{3}^{*}}, \dot{c})
&=
-a_{2} \left( x_{1} \dot{x}_{4} + \frac{x_{1} x_{4} \dot{x}_{5}}{1 - x_{5}} \right), \\
b_{123} (\tilde{E_{2}^{*}}, \tilde{E_{3}^{*}}, \dot{c}) 
&= - \frac{x_{1} x_{4} a_{2}}{2(1-x_{5})} \frac{d (r^{2})}{dt}, 
\end{align*}
and obtain (\ref{SO3 eq2}) and (\ref{SO3 eq3}).
\end{proof}

Next, we solve (\ref{SO3 eq1}), (\ref{SO3 eq2}), (\ref{SO3 eq3}). 
Calculating 
$(\ref{SO3 eq2}) + x_{4} \cdot (\ref{SO3 eq3})$, we have 
\begin{align} \label{SO3 eq4}
4 \dot{x}_{5} + x_{5} \frac{d}{dt} \log (\lambda + r^{2}) = 0.
\end{align}
Substitution of (\ref{SO3 eq4}) into (\ref{SO3 eq3}) gives 
\begin{align} \label{SO3 eq5}
4 \dot{x}_{4} + x_{4} \frac{d}{dt} \log (\lambda + r^{2}) = 0.
\end{align}
From (\ref{SO3 eq4}) and (\ref{SO3 eq5}), we have 
\begin{align*}
(1-x_{1}^{2}) \frac{d}{dt} \log (\lambda + r^{2}) 
&=
(x_{4}^{2} + x_{5}^{2}) \frac{d}{dt} \log (\lambda + r^{2}) \\
&=
-4 (x_{4} \dot{x}_{4} + x_{5} \dot{x}_{5}) \\
&= 4 x_{1} \dot{x}_{1}.
\end{align*}
which implies that (\ref{SO3 eq1}) is equivalent to 
\begin{align} \label{SO3 eq6}
x_{1} \frac{d}{dt} (a_{1} x_{1}) = 0.
\end{align}
Equations (\ref{SO3 eq4}), (\ref{SO3 eq5}), (\ref{SO3 eq6})
are solved easily and we obtain 
\begin{align*}
x_{4}^{4} (\lambda + r^{2}) = C, \qquad
x_{5}^{4} (\lambda + r^{2}) = D, \qquad
a_{1} x_{1} = E
\end{align*}
for $C,D \geq 0, E \in \mathbb{R}$. 
Thus 
\begin{align*}
M_{C, D, E} = 
{\rm SO}(3) \cdot 
\left \{ 
({}^t\! (x_{1}, 0, 0, x_{4}, x_{5}), {}^t\! (a_{1}, a_{2}, 0)); 
\begin{array}{c}
x_{4}^{4} (\lambda + r^{2}) = C,\\
x_{5}^{4} (\lambda + r^{2}) = D,\\
a_{1} x_{1} = E
\end{array}
\right \}
\end{align*}
is a coassociative submanifold for $C,D \geq 0, E \in \mathbb{R}$. 

Next, we consider the topology of $M_{C, D, E}$.

\begin{lem} \label{topology of M CDE}
Set 
$N= (\mathbb{R}_{\geq 0} \times {\rm SO}(3))/(\{ 0 \} \times {\rm SO}(3))$,
which is the cone over ${\rm SO}(3)$ with the apex. 
Then the topology of $M_{C,D,E}$ is given by the following. 

\begin{tabular}{|c|c|} 
\hline
condition & topology of $M_{C,D,E}$ \\
\hline \hline
$C>0, D>0, E=0, \sqrt{C} + \sqrt{D} \neq \sqrt{\lambda}$ 
& $TS^{2} \sqcup  TS^{2} \sqcup TS^{2} \sqcup TS^{2}$ \\ \hline
$C>0, D>0, E=0, \sqrt{C} + \sqrt{D} = \sqrt{\lambda}$ 
& $N \sqcup N \sqcup N \sqcup N$
 \\ \hline
$C>0, D>0, E \neq 0$ 
& $TS^{2} \sqcup  TS^{2} \sqcup TS^{2} \sqcup TS^{2}$ \\ \hline
$C=0, D=0$ 
& $TS^{2}$ \\ \hline
$C>0, D=0, E=0, \sqrt{C} \neq \sqrt{\lambda}$ 
& $TS^{2} \sqcup  TS^{2}$ \\ \hline
$C>0, D=0, E=0, \sqrt{C} = \sqrt{\lambda}$ 
& $N \sqcup N$ \\ \hline
$C>0, D=0, E \neq 0$ 
& $TS^{2} \sqcup  TS^{2}$ \\ \hline
$C=0, D>0, E=0, \sqrt{D} \neq \sqrt{\lambda}$ 
& $TS^{2} \sqcup  TS^{2}$ \\ \hline
$C=0, D>0, E=0, \sqrt{D} = \sqrt{\lambda}$ 
& $N \sqcup N$ \\ \hline
$C=0, D>0, E \neq 0$ 
& $TS^{2} \sqcup  TS^{2}$ \\
\hline
\end{tabular}
\end{lem}

\begin{lem}\label{converge SO3 current}
For any convergent sequence $\{ (C_{j}, D_{j}) \} \subset (\mathbb{R}_{>0})^{2}$
satisfying $\sqrt{C_{j}}+\sqrt{D_{j}} < \sqrt{\lambda}$ for any $j$ 
(or $\sqrt{C_{j}}+\sqrt{D_{j}} > \sqrt{\lambda}$ for any $j$) 
and $\sqrt{C_{\infty}}+\sqrt{D_{\infty}} = \sqrt{\lambda}$, 
where $C_{\infty} = \lim_{j \to \infty} C_{j}, D_{\infty} = \lim_{j \to \infty} D_{j}$, 
$M_{C_{j}, D_{j}, 0}$ converges to $M_{C_{\infty}, D_{\infty}, 0}$
in the sense of currents. 

Similarly, for any convergent sequence $\{ C_{j}\} \subset \mathbb{R}_{>0}$
satisfying $\sqrt{C_{j}} < \sqrt{\lambda}$ for any $j$ 
(or $\sqrt{C_{j}} > \sqrt{\lambda}$ for any $j$) 
and $\sqrt{C_{\infty}} = \sqrt{\lambda}$, 
where $C_{\infty} = \lim_{j \to \infty} C_{j}$, 
$M_{C_{j}, 0, 0}$ converges to $M_{C_{\infty}, 0, 0}$ 
and 
$M_{0, C_{j}, 0}$ converges to $M_{0, C_{\infty}, 0}$ 
in the sense of currents. 
\end{lem}

\begin{proof}[Proof of Lemma \ref{topology of M CDE}]
First, suppose that $M_{C,D,E}$ does not intersect 
with $\Lambda^{2}_{-} S^{4} |_{{}^t\! (0,0,0,0, \pm 1)}$. 
Then 
by (\ref{SO3 times 1 action}) we see that
\begin{align*}
M_{C,D,E} =
\left \{ \left(
\left( 
\begin{array}{c}
g_{11} x_{1}\\
g_{21} x_{1} \\
g_{31} x_{1} \\
x_{4} \\
x_{5}
\end{array}
\right), 
\left( 
\begin{array}{c}
a_{1} g_{11} + a_{2} g_{12}\\
a_{1} g_{21} + a_{2} g_{22}\\
-a_{1} g_{31} - a_{2} g_{32} 
\end{array}
\right)
\right)
;
\begin{array}{c}
x_{4}^{4} (\lambda + r^{2}) = C,\\
x_{5}^{4} (\lambda + r^{2}) = D,\\
a_{1} x_{1} = E, \\
(g_{ij}) \in {\rm SO}(3)
\end{array}
\right \} \\
=
\left \{ \left(
\left( 
\begin{array}{c}
x_{1}\\
x_{2} \\
x_{3} \\
x_{4} \\
x_{5}
\end{array}
\right), 
\left( 
\begin{array}{c}
a_{1}\\
a_{2}\\
a_{3}
\end{array}
\right)
\right)
\in S^{4} \times \mathbb{R}^{3};
\begin{array}{c}
x_{4}^{4} (\lambda + r^{2}) = C,\\
x_{5}^{4} (\lambda + r^{2}) = D,\\
(r^{2} = \sum_{i=1}^{3} a_{i}^{2})\\
a_{1} x_{1} + a_{2} x_{2} - a_{3} x_{3} = E \\
\end{array}
\right \}.
\end{align*}

We study the topology of $M_{C,D,E}$
in the following cases: 
\begin{enumerate}
\item $C>0, D>0$, 
\begin{multicols}{2}
\begin{enumerate}
\item $E=0, \sqrt{C} + \sqrt{D} < \sqrt{\lambda}$, 
\item $E=0, \sqrt{C} + \sqrt{D} > \sqrt{\lambda}$, 
\item $E=0, \sqrt{C} + \sqrt{D} = \sqrt{\lambda}$, 
\item $E \neq 0$, 
\end{enumerate}
\end{multicols}
\begin{multicols}{3}
\item $C=0,D=0$,
\item $C>0, D=0$,
\item $C=0,D>0$.
\end{multicols}
\end{enumerate}

Consider \underline{case 1}.
Then $M_{C,D,E}$ does not intersect 
with $\Lambda^{2}_{-} S^{4} |_{{}^t\! (0,0,0,0, \pm 1)}$.
Set 
\begin{align*}
M_{C,D,E}^{\pm, +} &= M_{C,D,E} \cap \{ \pm x_{4} > 0 \} \cap \{ x_{5} > 0\}, \\
M_{C,D,E}^{\pm, -} &= M_{C,D,E} \cap \{ \pm x_{4} > 0 \} \cap \{ x_{5} < 0\}. 
\end{align*}
Each $M_{C,D,E}^{\pm, \pm}$ is a connected component of $M_{C,D,E}$ 
and is homeomorphic to 
\begin{align} \label{def N CDE}
N_{C,D,E} = \left \{ (v, w) \in \mathbb{R}^{3} \times \mathbb{R}^{3}; \langle v, w \rangle = E, 
(1-|v|^{2}) \sqrt{\lambda + |w|^{2}} = \sqrt{C} + \sqrt{D} \right \}.
\end{align}
We only have to consider the topology of $N_{C,D,E}$. 

Consider \underline{case 1-(a)}.
We have 
$|v|^{2} = 1- (\sqrt{C} + \sqrt{D})/ \sqrt{\lambda + |w|^{2}} \geq  1- (\sqrt{C} + \sqrt{D})/ \sqrt{\lambda} > 0$.
Hence there is an homeomorphism 
$N_{C,D,0} \rightarrow \{ (v,w) \in S^{2} \times \mathbb{R}^{3}; \langle v, w \rangle = 0\} = TS^{2}$ 
via $(v, w) \mapsto (v/|v|, w)$.

Consider \underline{case 1-(b)}.
We have 
$|w|^{2} = (\sqrt{C} + \sqrt{D})^{2}/(1 - |v|^{2})^{2} - \lambda 
\geq (\sqrt{C} + \sqrt{D})^{2} - \lambda > 0$.
Hence there is an homeomorphism 
$N_{C,D,0} \rightarrow \{ (w,v) \in S^{2} \times \mathbb{R}^{3}; \langle w, v \rangle = 0, |v| < 1 \} \cong TS^{2}$ 
via $(v, w) \mapsto (w/|w|, v)$. 

Consider \underline{case 1-(c)}.
A map  
$N= (\mathbb{R}_{\geq 0} \times {\rm SO}(3))/(\{ 0 \} \times {\rm SO}(3)) \rightarrow N_{C,D,0}$ 
defined by 
$[(r, (g_{1}, g_{2}, g_{3})] \mapsto (f(r) g_{1}, r g_{2})$, 
where $g_{i} \in \mathbb{R}^{3}$, $\langle g_{i}, g_{j} \rangle = \delta_{i j}$, 
and $f(r) = \sqrt{1-(\sqrt{C} + \sqrt{D})/\sqrt{\lambda + r^{2}}}$, 
gives a homeomorphism.

Consider \underline{case 1-(d)}.
Since $N_{C,D,E} \cong N_{C,D, -E}$ via $(v,w) \mapsto (v, -w)$, 
we may assume that $E>0$. 
Since $E \neq 0$, we have $v, w \neq 0$ for any $(v, w) \in N_{C,D,E}.$
Define $c_{0} \in \mathbb{R}$ and a function $f: (c_{0}, \infty) \rightarrow (f(c_{0}), 1)$ by 
\begin{align*}
c_{0} &= 
\left\{ \begin{array}{ll}
0                                                      & \qquad \mbox{when }\ (\sqrt{C} + \sqrt{D})^{2} - \lambda \leq 0, \\
\sqrt{(\sqrt{C} + \sqrt{D})^{2} - \lambda} & \qquad \mbox{when }\ (\sqrt{C} + \sqrt{D})^{2} - \lambda \geq 0
\end{array} \right.,\\
f(r) &= \sqrt{1-\frac{\sqrt{C} + \sqrt{D}}{\sqrt{\lambda + r^{2}}}}.
\end{align*}
Then 
$f$ is bijective and monotonically increasing. 
Note that for $(v, w) \in N_{C,D,E}$, we have $f(|w|) = |v|$. 
Since $r f(r): (c_{0}, \infty) \rightarrow (0, \infty)$ is bijective and monotonically increasing, 
there exists a unique $d_{0} > c_{0} > 0$ such that $d_{0} f(d_{0}) = E$. 
Now define a function $g: [d_{0}, \infty) \rightarrow [0, \infty)$ by 
$g(r)=\sqrt{r^{2} - (E^{2}/f(r)^{2})}$. 
Note that 
for $(v, w) \in N_{C,D,E}$, we have $|w- (Ev/|v|^{2})| = g(|w|)$. 

Define a map 
$\Phi :N_{C,D,E} \rightarrow \{ (v', w') \in S^{2} \times \mathbb{R}^{3}; \langle v, w \rangle = 0 \} = TS^{2}$
by 
$\Phi(v, w) = (v/|v|, w- (Ev/|v|^{2}))$. 
Then $\Phi$ is a homeomorphism and the inverse map 
$\Phi^{-1}$ is given by $\Phi^{-1}(v', w') = (f(g^{-1}(|w'|))v', w'+ (Ev'/f(g^{-1}(|w'|))))$.

Consider \underline{case 2}.
By definition, we have $x_{4} = x_{5} = 0$. Then 
\begin{align*}
M_{0,0,E} =
\left \{ \left(
{}^t\! \left( 
x_{1}, x_{2}, x_{3}, 0, 0
\right), 
{}^t\!
\left( 
a_{1}, a_{2}, a_{3}
\right)
\right)
\in S^{4} \times \mathbb{R}^{3};
a_{1} x_{1} + a_{2} x_{2} - a_{3} x_{3} = E
\right \}, 
\end{align*}
which is obtained in (\ref{SO3SO2 case2 MC}) 
and is homeomorphic to $TS^{2}$.

Consider \underline{case 3}.
By definition, we have $x_{5} = 0$ and 
\begin{align*}
M_{C,0,E} =
\left \{ \left(
{}^t\! \left( 
x_{1}, x_{2}, x_{3}, x_{4}, 0
\right), 
{}^t\!
\left( 
a_{1}, a_{2}, a_{3}
\right)
\right)
\in S^{4} \times \mathbb{R}^{3};
\begin{array}{c}
x_{4}^{4} (\lambda + r^{2}) = C, \\
a_{1} x_{1} + a_{2} x_{2} - a_{3} x_{3} = E
\end{array}
\right \}. 
\end{align*}
Set 
$
M_{C,0,E}^{\pm} = M_{C,0,E} \cap \{ \pm x_{4} > 0 \}. 
$
Each $M_{C,0,E}^{\pm}$ is a connected component of $M_{C,0,E}$ 
and is homeomorphic to 
$N_{C,0,E}$ defined in (\ref{def N CDE}).

Consider \underline{case 4}.
By (\ref{1 times SO2 action}), 
$
\left(
\begin{array}{cc}
0 & 1 \\
-1&0
\end{array}
\right) \in {\rm SO}(2) = \{ I_{3} \} \times {\rm SO}(2) \subset {\rm SO}(5)$ 
gives a homeomorphism from $N_{0,D,E}$ to $N_{D,0,E}$. 
Hence this case is reduced to \underline{case 3}. 
\end{proof}

\begin{proof}[Proof of Lemma \ref{converge SO3 current}]
We only have to prove that 
$N_{C_{j}, D_{j}, 0}$ converges to $N_{C_{\infty}, D_{\infty}, 0} - \{ (0, 0) \}$ 
in the sense of currents. 
Note that sets differing only a set of measure zero are identified 
in the theory of currents.

Suppose that 
$\sqrt{C_{j}}+\sqrt{D_{j}} < \sqrt{\lambda}$ for any $j$. 
Then by the proof of Lemma \ref{topology of M CDE}, 
there is a homeomorphism 
$h_{C_{j}, D_{j}}: N_{C_{j},D_{j},0} \rightarrow \{ (v,w) \in S^{2} \times \mathbb{R}^{3}; \langle v, w \rangle = 0\} = TS^{2}$ 
via $(v, w) \mapsto (v/|v|, w)$.
Note that $h_{C_{j}, D_{j}}^{-1}$ is given by 
$(v', w') \mapsto (f_{C_{j}, D_{j}}(|w'|)v', w')$, 
where $f_{C_{j}, D_{j}} (r) = \sqrt{1-(\sqrt{C_{j}} + \sqrt{D_{j}})/\sqrt{\lambda + r^{2}}}$.

On the other hand, 
$N_{C_{\infty}, D_{\infty}, 0} - \{ (0, 0) \}$ is homeomorphic to 
$T S^{2} - \{ 0 \}$ 
via $h_{C_{\infty}, D_{\infty}}: (v, w) \mapsto (v/|v|, w)$ 
and $h_{C_{\infty}, D_{\infty}}^{-1}$ is given by 
$(v', w') \mapsto (f_{C_{\infty}, D_{\infty}}(|w'|)v', w')$. 
Then we see that for any compactly supported 4-form $\alpha$ on $\mathbb{R}^{3} \times \mathbb{R}^{3}$
\begin{align*}
\int_{N_{C_{j}, D_{j}, 0}} \alpha 
=
\int_{TS^{2}-\{ 0 \}} (h_{C_{j}, D_{j}}^{-1})^{*} \alpha
\rightarrow
\int_{TS^{2}-\{ 0 \}} (h_{C_{\infty}, D_{\infty}}^{-1})^{*} \alpha
=
\int_{N_{C_{\infty}, D_{\infty}, 0} - \{ (0, 0) \}} \alpha, 
\end{align*}
which implies that 
$N_{C_{j}, D_{j}, 0}$ converges to $N_{C_{\infty}, D_{\infty}, 0} - \{ (0, 0) \}$ 
in the sense of currents. 
We can prove the other cases similarly 
and obtain the statement. 
\end{proof}

\begin{rem}
Use the notation in \cite{Lotay3}. 
By Lemma \ref{topology of M CDE}, 
$M_{C,D,E}$ is a coassociative submanifold 
with conical singularities 
when 
(i) $C>0, D>0, E=0, \sqrt{C} + \sqrt{D} = \sqrt{\lambda}$, 
(ii)$C>0, D=0, E=0, \sqrt{C} = \sqrt{\lambda}$, or 
(iii)$C=0, D>0, E=0, \sqrt{D} = \sqrt{\lambda}$. 
In each case, 
the tangent cone is modeled on $C(L) = \mathbb{R}_{>0} \times L$, 
where $L$ is given by 
\begin{align*}
L = \left \{ {}^t\! (0, z_{1}, z_{2}, z_{3}) \in \mathbb{R} \oplus \mathbb{C}^{3}; 
z_{1}^{2} + z_{2}^{2} + \overline{z}_{3}^{2} = 0, |z_{1}|^{2} + |z_{2}|^{2} + |z_{3}|^{2} = 1 \right \} 
\cong {\rm SO}(3). 
\end{align*}

We calculate the rate at singular points as follows. 
For simplicity, we only consider the case of $M^{+}_{\lambda, 0 ,0}$ 
which is singular at $p_{0} = \left( {}^t\! (0, 0, 0, 1, 0), {}^t\! (0, 0, 0) \right)$. 

Let $B(0, r) \subset \mathbb{R}^{4}$ be an open ball of radius $r$.  
Set $D = \{x_{4} > 0 \} \subset S^{4}$ and $k=2^{-1/2} \lambda^{-1/4}$.
Define $\chi: B(0, 1/k) \times \mathbb{R}^{3} \rightarrow D \times \mathbb{R}^{3}$ by 
\begin{align*}
\left( {}^t\! (u_{1}, u_{2}, u_{3}, u_{4}), {}^t\! (a_{1},  a_{2}, a_{3}) \right) \mapsto
\left(
{}^t\! (-k u_{3}, k u_{2}, -k u_{1}, \sqrt{1-k^{2} |u|^{2}}, k u_{4}), 
\lambda^{1/4} {}^t\! (a_{1},  a_{2}, a_{3})
\right),
\end{align*}
where $|u|^{2}=\sum_{j=1}^{4} u_{j}^{2}$. 
Since $(d \chi)_{0} (\frac{\partial}{\partial u_{i}})_{0} = k (e_{i})_{p_{0}}, 
(d \chi)_{0} (\frac{\partial}{\partial a_{i}})_{0} = \lambda^{1/4} (\frac{\partial}{\partial a_{i}})_{p_{0}}$, 
we see that 
$(d \chi)_{0}^{*} (\varphi_{\lambda})_{p_{0}} = \varphi_{0}$, 
where $\varphi_{0}$ is a 3-form on $\mathbb{R}^{7}$ given by (\ref{def of G2 str}). 
Note that 
\begin{align*}
\chi^{-1} (M^{+}_{\lambda, 0 ,0})
=
\left \{
{}^t\! (u_{1}, u_{2}, u_{3}, 0, a_{1},  a_{2}, a_{3}); 
\begin{array}{l}
(1-k^{2}|u|^{2})^{2} \left(1+ \sum_{j=1}^{3} a_{j}^{2} \right) = 1, \\
-u_{3} a_{1} + u_{2} a_{2} - u_{1} a_{3} = 0
\end{array}
\right \}.
\end{align*}
Define $\Phi: \mathbb{R}_{>0} \times L \rightarrow \mathbb{R}^{7}$ by 
\begin{align*}
\left( r, {}^t\! (0, x_{1}+ i y_{1}, x_{2}+ i y_{2}, x_{3}+ i y_{3}) \right) 
\mapsto 
{}^t\! \left( f(r) x_{1}, f(r) x_{2}, f(r) x_{3}, 0, r y_{3},  r y_{2}, - r y_{1} \right),  
\end{align*}
where $f(r) = 2 \lambda^{1/4} \sqrt{1- \sqrt{\frac{2}{2+r^{2}}}}$. 
This gives the diffeomorphism $\chi \circ \Phi: \mathbb{R}_{>0} \times L \rightarrow 
M^{+}_{\lambda, 0 ,0} - \{ p_{0} \}$. 
Since we see that 
$f(r) = \lambda^{1/4} r + O(r^{3})$ as $r \rightarrow 0$, we see that 
the rate at $p_{0}$ is equal to 3 in these coordinates.
\end{rem}


\subsection{Irreducible ${\rm SO}(3)$-action} \label{cohomo1 section irr SO3}

We give a proof of Theorem \ref{irr SO3 thm}.
Recall the notation in Section \ref{section irr SO3}. 
By Lemma \ref{orbit irr SO3}, 
an ${\rm SO}(3)$-orbit 
through $({}^t\! (x_{1}, 0, 0, 0, x_{5}), {}^t\! (a_{1}, a_{2}, a_{3}))$ 
is 3-dimensional when 
\begin{enumerate}
\item
$-1/2 < x_{5} < 1/2$, 
\item
$x_{5} = 1/2, (a_{1}, a_{3}) \neq 0$, \mbox{ or }
\item
$x_{5} = -1/2, (a_{2}, a_{3}) \neq 0$.
\end{enumerate}

Consider \underline{case 1}. 
Take a path $c: I \rightarrow \Lambda^{2}_{-} S^{4}$ given by 
\begin{align*}
c(t) = \left( {}^t\! (x_{1}(t), 0, 0, 0, x_{5}(t)), {}^t\! (a_{1}(t), a_{2}(t), a_{3}(t)) \right), 
\end{align*}
where $x_{1}(t) > 0, |x_{5}(t)| < 1/2$. 
We find a path $c$ satisfying $\varphi_{\lambda} |_{{\rm SO}(3) \cdot {\rm Image} (c)} = 0$, 
where $\varphi_{\lambda}$ is given by (\ref{def of G2 str S4}).  
We see that 
$\varphi_{\lambda} (\tilde{E_{1}^{*}}, \tilde{E_{2}^{*}}, \tilde{E_{3}^{*}})=0$ 
as in Section \ref{cohomo1 section SU2}.

\begin{lem}
The condition $\varphi_{\lambda} |_{{\rm SO}(3) \cdot {\rm Image} (c)} = 0$ 
is equivalent to
\begin{align}
4 \left \{
(2 \sqrt{3} x_{1} + 4 x_{5} + 1) \dot{x}_{5} + \sqrt{3} (2 x_{5} -1) \dot{x}_{1} 
\right \} a_{1}
+ 8 x_{1} (- x_{1} + \sqrt{3} x_{5}) \dot{a}_{1}& \nonumber \\
- (\sqrt{3} x_{1} + x_{5} +1)(1- 2 x_{5}) a_{1} \frac{d}{dt} \log (\lambda + r^{2}) &= 0, 
\label{irr SO3 eq1}\\
4 \left \{
(2 \sqrt{3} x_{1} - 4 x_{5} - 1) \dot{x}_{5} + \sqrt{3} (2 x_{5} -1) \dot{x}_{1} 
\right \} a_{2}
+ 8 x_{1} (x_{1} + \sqrt{3} x_{5}) \dot{a}_{2}& \nonumber \\
+ (- \sqrt{3} x_{1} + x_{5} + 1)(1- 2 x_{5}) a_{2} \frac{d}{dt} \log (\lambda + r^{2}) &= 0, 
\label{irr SO3 eq2}\\
4 \left \{
 - (x_{5} + 1) \dot{x}_{5} + 3 x_{1} \dot{x}_{1} 
\right \} a_{3} 
+ 2 (x_{1}^{2} - 3 x_{5}^{2}) \dot{a}_{3}& \nonumber \\
+ (1 + x_{5})(1- 2 x_{5}) a_{3} \frac{d}{dt} \log (\lambda + r^{2}) &= 0, \label{irr SO3 eq3}
\end{align}
where $r^{2} = \sum_{j=1}^{3} a_{j}^{2}$.
\end{lem}
This lemma implies Theorem \ref{irr SO3 thm}.
In general, it is hard to solve 
the equations (\ref{irr SO3 eq1}), (\ref{irr SO3 eq2}), (\ref{irr SO3 eq3}) explicitly.

\begin{proof}
Since 
$
\dot{c} = 
(-\dot{x}_{1} x_{5} + x_{1} \dot{x}_{5}) e_{4} 
+ \sum_{j=1}^{3} \dot{a}_{j} \frac{\partial}{\partial a_{j}},  
$
we have 
\begin{align*}
(\pi^{*} \omega_{i} (\tilde{E_{j}^{*}}, \dot{c})) 
&= 
\left( 
 \begin{array}{ccc}
-\dot{x}_{5} + \sqrt{3} \dot{x}_{1}   & 0                     & 0  \\
0                                            & 0  & \dot{x}_{5} + \sqrt{3} \dot{x}_{1} \\
0                                            & -2\dot{x}_{5}  & 0 
\end{array}
\right), \\
b_{j} (\dot{c})
&= a_{j} \qquad \mbox{ for } j=1,2,3. 
\end{align*}
Then we compute 
\begin{align*}
\sum_{i=1}^{3} b_{i} \wedge \pi^{*} \omega_{i} (\tilde{E_{1}^{*}}, \tilde{E_{2}^{*}}, \dot{c})
=&
\left \{
(2 \sqrt{3} x_{1} - 4 x_{5} - 1) \dot{x}_{5} + \sqrt{3} (2 x_{5} -1) \dot{x}_{1} 
\right \} a_{2}\\
&+ 2 x_{1} (x_{1} + \sqrt{3} x_{5}) \dot{a}_{2}, \\
b_{123} (\tilde{E_{1}^{*}}, \tilde{E_{2}^{*}}, \dot{c}) 
=& - (\sqrt{3} x_{1} - x_{5} -1)(1- 2 x_{5}) \frac{a_{2}}{2} \frac{d(r^{2})}{dt}, 
\end{align*}
which implies (\ref{irr SO3 eq2}). In the same way, we compute 
\begin{align*}
\sum_{i=1}^{3} b_{i} \wedge \pi^{*} \omega_{i} (\tilde{E_{1}^{*}}, \tilde{E_{3}^{*}}, \dot{c})
=&
\left \{
-2 (x_{5} + 1) \dot{x}_{5} + 6 x_{1} \dot{x}_{1} 
\right \} a_{3}
+ (x_{1}^{2} - 3 x_{5}^{2}) \dot{a}_{3}, \\
b_{123} (\tilde{E_{1}^{*}}, \tilde{E_{3}^{*}}, \dot{c}) 
=& (1 + x_{5})(1- 2 x_{5}) a_{3} \frac{d(r^{2})}{dt}, \\
\sum_{i=1}^{3} b_{i} \wedge \pi^{*} \omega_{i} (\tilde{E_{2}^{*}}, \tilde{E_{3}^{*}}, \dot{c})
=&
\left \{
(2 \sqrt{3} x_{1} + 4 x_{5} + 1) \dot{x}_{5} + \sqrt{3} (2 x_{5} -1) \dot{x}_{1} 
\right \} a_{1}\\
&+ 2 x_{1} (\sqrt{3} x_{5} - x_{1}) \dot{a}_{1}, \\
b_{123} (\tilde{E_{1}^{*}}, \tilde{E_{2}^{*}}, \dot{c}) 
=& - (\sqrt{3} x_{1} + x_{5} + 1)(1- 2 x_{5}) \frac{a_{1}}{2} \frac{d(r^{2})}{dt}, 
\end{align*}
and obtain (\ref{irr SO3 eq1}) and  (\ref{irr SO3 eq3}).
\end{proof}

Consider \underline{case 2}.
Take a path $c: I \rightarrow \Lambda^{2}_{-} S^{4}$ given by 
\begin{align*}
c(t) = \left( {}^t\! (\sqrt{3}/2, 0, 0, 0, 1/2), {}^t\! (a_{1}(t), a_{2} (t), a_{3} (t)) \right). 
\end{align*}
We may assume that $a_{3}=0$ so that 
$c(t)$ is transverse to the ${\rm SO}(3)$-orbit. 
We find a path $c$ satisfying $\varphi_{\lambda}|_{{\rm SO}(3) \cdot {\rm Image} (c)} = 0$, 
where $\varphi_{\lambda}$ is given by (\ref{def of G2 str S4}).  
Since 
$\dot{c} = \sum_{i=1}^{2} \dot{a}_{i} \frac{\partial}{\partial a_{i}}$, 
we have at $c(t)$
\begin{align*}
(\pi^{*} \omega_{i} (\tilde{E_{j}^{*}}, \dot{c})) 
&= 0, \\
(b_{i} (\dot{c})) 
&= 
\left( 
\dot{a}_{1}, \dot{a}_{2}, 0
\right), \\
\varphi_{\lambda} (\tilde{E_{1}^{*}}, \tilde{E_{2}^{*}}, \dot{c}) 
&= 6 s_{\lambda} \dot{a}_{2}, \\
\varphi_{\lambda} (\tilde{E_{p}^{*}}, \tilde{E_{q}^{*}}, \dot{c}) &= 0
\qquad \mbox{ for }\ (p,q) \neq (1,2),(2,1). 
\end{align*}
Thus 
the condition $\varphi_{\lambda} |_{{\rm SO}(3) \cdot {\rm Image} (c)} = 0$ is equivalent to 
$a_{2}=C$ for $C \in \mathbb{R}$. 
Then as Remark \ref{relation Kari_Min}, we see that 
$M_{C} = {\rm SO}(3) \cdot \{ ( {}^t\! (\sqrt{3}/2, 0, 0, 0, 1/2), {}^t\! (r, C, 0)); r \in \mathbb{R} \}$, 
where $C \in \mathbb{R}$,  
is a coassociative submanifold 
described as 
\begin{align*}
M_{C} = C \tau + (\mathbb{R} \tau)^{\perp}, 
\end{align*}
where 
$\tau = {\rm vol}_{\Sigma} - * {\rm vol}_{\Sigma}$ 
and $\Sigma = {\rm SO}(3) \cdot  {}^t\! (\sqrt{3}/2, 0, 0, 0, 1/2) \subset S^{4}$
is a Veronese surface. 
In \underline{Case 3}, we obtain the similar coassociative submanifold, 
and hence we cannot obtain new examples in \underline{Case 2} and \underline{Case 3}.


\subsection{Cohomogeneity two coassociative submanifolds}
When $\lambda \rightarrow 0$, 
$\varphi_{0} = \varphi_{\lambda}|_{\lambda=0}$ defines a 
$G_{2}$-structure on $\Lambda^{2}_{-} S^{4} - \{ \mbox{zero section} \}  \cong 
\mathbb{C} P^{3} \times \mathbb{R}_{>0}$ by Remark \ref{cone_of_CP3}. 
On $\Lambda^{2}_{-} S^{4} - \{ \mbox{zero section} \}$, 
$\mathbb{R}_{>0}$ acts by dilations preserving $\varphi_{0}$ 
up to scalar multiplication. 
Thus by using the $\mathbb{R}_{>0}$-action, 
we can apply the same method as the cohomogeneity one case 
and we derive some systems of O.D.E.s. 
However, we can find no explicit solutions which 
give new coassociative examples. 
In some cases, 
we obtain some explicit solutions, 
all of which turn out to be congruent to examples in Section \ref{cohomo one section} 
up to the ${\rm SO}(5)$-action.

\appendix
\section{Real irreducible representations} \label{real irr rep}

We give a summary about real irreducible representations in 
\cite{GotoGrosshans, Mashimo minimal}. 

\begin{definition}
Let $G$ be a compact Lie group and 
$(V, \rho)$ be a $\mathbb{C}$-irreducible representation of $G$. 
We call $(V, \rho)$ {\bf self-conjugate} if 
$V$ has a conjugate linear map $J$ on $V$ satisfying 
\begin{align*}
J^{2} = \pm 1, \qquad 
J \circ \rho(g) = \rho(g) \circ J \ \mbox{ for } g \in G. 
\end{align*}
This map is called a {\bf structure map}. 
A self-conjugate representation $(V, \rho)$ is said to be of 
{\bf index} $\pm 1$
if $J^{2}= \pm 1$.
\end{definition}

\begin{prop} \label{C-irr rep classification}
Let $(V, \rho)$ be a $\mathbb{C}$-irreducible representation of $G$. 
Then 
one of the following is satisfied. 
\begin{enumerate}
\item 
$(V, \rho)$ is a self-conjugate representation of index $1$. 
In this case, $(V, \rho)$ is a complexification of 
a real representation. 
\item 
$(V, \rho)$ is a self-conjugate representation of index $-1$. 
In this case, $(V, \rho)$ is a quaternionic representation. 
\item 
$(V, \rho)$ is not a self-conjugate representation. 
\end{enumerate}
\end{prop}

\begin{prop}
Let $(V, \rho)$ be a $\mathbb{C}$-irreducible representation of $G$. 
As a real representation, 
$\rho$ is reducible (resp. irreducible)
if and only if 
1. (resp. 2. or 3.) in Proposition \ref{C-irr rep classification} is satisfied.
\end{prop}

\begin{prop} \label{R-rep C-rep}
All $\mathbb{R}$-irreducible representations of $G$ are given as follows. 
\begin{itemize}
\item
A $\mathbb{R}$-irreducible component of 
a $\mathbb{C}$-irreducible representation which is reducible as a $\mathbb{R}$-representation. 
This is
an eigenspace of $1$ or $-1$ of the structure map $J$ in 1. of Proposition \ref{C-irr rep classification}. 
(Note that an eigenspace of $1$ and that of $-1$ are mutually equivalent real irreducible representations of $G$.)
\item
A $\mathbb{C}$-irreducible representation which is also irreducible as a $\mathbb{R}$-representation. 
This corresponds 2. or 3. in Proposition \ref{C-irr rep classification}.
\end{itemize}
\end{prop}
In many cases, we know $\mathbb{C}$-irreducible representations, 
from which 
we can deduce $\mathbb{R}$-irreducible representations by Proposition \ref{R-rep C-rep}. 

All equivalence classes of finite dimensional 
$\mathbb{C}$-irreducible representations of ${\rm SU}(2)$  
are represented by $\{ (V_{n}, \rho_{n}) \}_{n \geq 0}$, 
where $V_{n}$ is a $\mathbb{C}$-vector space of all complex 
homogeneous polynomials with two variables $z_{1}, z_{2}$ of degree $n$ 
and $\rho_{n}$ is the induced action from the standard action of ${\rm SU}(2)$ on $\mathbb{C}^{2}$. 
By Proposition \ref{R-rep C-rep}, we deduce the following. 

\begin{lem}[\cite{Mashimo minimal}] \label{R-rep su2} 
Let $V$ be a $\mathbb{R}$-irreducible representation of ${\rm SU}(2)$. 
Then $\dim_{\mathbb{R}}V = 4m$ or $2n-1$, where  $m,n \geq 1.$
\end{lem}

For compact Lie groups $H_{1}$ and $H_{2}$, 
any $\mathbb{C}$-irreducible representation of $H_{1} \times H_{2}$ 
is given by $\sigma_{1} \otimes \sigma_{2}$, 
where $\sigma_{i}$ is a irreducible $\mathbb{C}$-representation of $H_{i}$. 
Thus in the same way, we obtain the following.

\begin{lem} \label{R-rep su2 su2}
Let $V$ be a $\mathbb{R}$-irreducible representation of ${\rm SU}(2) \times {\rm SU}(2)$. 
Then 
\begin{align*}
\dim_{\mathbb{R}}V = 
\left\{
\begin{array}{ll}
2(k+1)(l+1) & \mbox{ when } k,l \geq 0, k+l \mbox{: odd},\\
(k+1)(l+1) &   \mbox{ when } k,l \geq 0, k+l \mbox{: even}.
\end{array}
\right.
\end{align*}
If $k=0$ or $l=0$, the representation reduces to that of ${\rm SU}(2)$. 
\end{lem}

\begin{lem} \label{R-rep su2 su2 su2}
Let $V$ be a $\mathbb{R}$-irreducible representation of 
${\rm SU}(2) \times {\rm SU}(2) \times {\rm SU}(2)$. 
Then 
\begin{align*}
\dim_{\mathbb{R}}V = 
\left\{
\begin{array}{ll}
2(k+1)(l+1)(m+1) & \mbox{ when } k,l,m \geq 0, k+l+m \mbox{: odd},\\
(k+1)(l+1)(m+1)   & \mbox{ when } k,l,m \geq 0, k+l+m \mbox{: even}.
\end{array}
\right.
\end{align*}
If one of $\{ k, l, m \}$ is equal to 0, the representation reduces to that of ${\rm SU}(2) \times {\rm SU}(2)$. 
If two of $\{ k, l, m \}$ are equal to 0, the representation reduces to that of ${\rm SU}(2)$. 
\end{lem}


\section{Proof of Lemma \ref{Lie subgrp of SO5}} \label{proof of classification subgrp SO(5)}

First, we prove the following.

\begin{lem} \label{Lie subalg of so5}
Let $\mathfrak{g} \subset \mathfrak{so}(5)$ be a compact Lie subalgebra with 
$\dim_{\mathbb{R}} \mathfrak{g} \geq 3$. 
Then $\mathfrak{g}$ is isomorphic to one of the following Lie algebras: 
\begin{align*}
\mathfrak{so}(5), \qquad \mathfrak{so}(4), \qquad 
\mathfrak{su}(2) \oplus \mathbb{R}, \qquad \mathfrak{su}(2). 
\end{align*}
\end{lem}

For the proof of Lemma \ref{Lie subalg of so5}, 
we need the $\mathbb{R}$-irreducible representations of compact Lie groups in Appendix \ref{real irr rep}. 
By Lemma \ref{Lie subalg of so5} and its proof, 
we obtain Lemma \ref{Lie subgrp of SO5}.

\begin{proof}
By the classification of compact Lie algebras, 
the possible $k$-dimensional Lie subalgebra $\mathfrak{g}$ of  $\mathfrak{so}(5)$, 
where $3\leq k \leq 10$, is isomorphic to one of the following: 
\begin{align*}
\begin{array}{ll}
\mathfrak{so}(5) & \mbox{ for } k=10,\\
\mathbb{R} \oplus \mathfrak{su}(3), \ \mathfrak{su}(2)^{3}     & \mbox{ for } k=9,\\
\mathfrak{su}(3), \ \mathbb{R}^{2} \oplus \mathfrak{su}(2)^{2} & \mbox{ for } k=8,\\
\mathbb{R} \oplus \mathfrak{su}(2)^{2}                             & \mbox{ for } k=7,
\end{array}
\qquad
\begin{array}{ll}
\mathfrak{su}(2)^{2}                            & \mbox{ for } k=6,\\
\mathbb{R}^{2} \oplus \mathfrak{su}(2)  & \mbox{ for } k=5,\\
\mathbb{R} \oplus \mathfrak{su}(2)      & \mbox{ for } k=4,\\
\mathfrak{su}(2)                                & \mbox{ for } k=3.
\end{array}
\end{align*}
We check whether the Lie subalgebras in this list 
are actually contained in $\mathfrak{so}(5)$.

First, we show that  $\mathfrak{su}(3), \mathbb{R} \oplus \mathfrak{su}(3) \not\subset \mathfrak{so}(5)$.
By Theorem 5.10 of \cite{Hall}, the dimension of the $\mathbb{C}$-irreducible representation  
of $\mathfrak{su}(3)$ is of the form 
\begin{align*}
\frac{1}{2} 
(m_{1} + 1)(m_{2} + 1)(m_{1} + m_{2} + 2), 
\end{align*}
where $m_{j} \in \mathbb{Z}_{\geq 0}.$
Since any representation of the compact Lie algebra $\mathfrak{su}(3)$ is 
completely reducible, 
we see that $\mathfrak{su}(3) \not\subset \mathfrak{so}(5)$ by Proposition \ref{R-rep C-rep}, 
which implies that $\mathbb{R} \oplus \mathfrak{su}(3) \not\subset \mathfrak{so}(5)$. 

Similarly, by Lemma \ref{R-rep su2 su2 su2}, 
we see that $\mathfrak{su}(2)^{3} \not\subset \mathfrak{so}(5)$. 
By Lemma \ref{R-rep su2 su2}, 
the only inclusion $\mathfrak{su}(2)^{2} \hookrightarrow \mathfrak{so}(5)$ is 
the standard inclusion $\mathfrak{su}(2)^{2} = \mathfrak{so}(4) \hookrightarrow \mathfrak{so}(5)$. 
We may assume that 
$\mathfrak{so}(4) = 
\left(
\begin{array}{cc}
\mathfrak{so}(4) & \\
                       & 0
\end{array} 
\right) 
\hookrightarrow \mathfrak{so}(5)$.
Since 
\begin{align*}
\left\{ 
Y \in \mathfrak{so}(5); 
[X, Y] = 0 \mbox{ for any } X \in \mathfrak{so}(4)
\right \}
=
\{ 0 \},
\end{align*}
we see that 
$\mathbb{R}^{2} \oplus \mathfrak{su}(2)^{2}, \mathbb{R} \oplus \mathfrak{su}(2)^{2} 
\not\subset \mathfrak{so}(5)$.

By Lemma \ref{R-rep su2}, we have 3 types of inclusions 
$\mathfrak{su}(2) \hookrightarrow \mathfrak{so}(5)$ given by 
\begin{align}
&\mathfrak{su}(2) = \mathfrak{so}(3) \hookrightarrow \mathfrak{so}(5), \label{incl su2=so3}\\
&\mathfrak{su}(2) \hookrightarrow \mathfrak{so}(4) \hookrightarrow \mathfrak{so}(5), \label{incl diag su2}\\
&\mathfrak{su}(2) \hookrightarrow \mathfrak{so}(5) \mbox {: irreducibly}. \label{incl irr su2}
\end{align}
Note that the basis of $\mathfrak{su}(2)$ of (\ref{incl su2=so3}) is given by 
$
\left \{
\left(
\begin{array}{cc}
E_{i} &       \\
      & O_{2}
\end{array} 
\right)
\right \}_{i=1,2,3},
$
where $E_{i}$ is defined in (\ref{basis of so3}).
The basis of $\mathfrak{su}(2)$ of (\ref{incl diag su2}) is given by 
(\ref{lie alg of diag su2}), 
and that of  (\ref{incl irr su2}) is given by 
(\ref{lie alg of irr su2}). 
We easily see that 
$Z:= \{ Y \in \mathfrak{so}(5); [X, Y] =0 \mbox{ for any } X \in \mathfrak{su}(2) \}$ 
is spanned by 
\begin{align*}
\left(
\begin{array}{cc}
O_{3} &     \\
       & J   
\end{array} 
\right) 
\mbox{ for (\ref{incl su2=so3}), } 
O_{5}
\mbox{ for (\ref{incl irr su2}), } \\
\left(
\begin{array}{ccc}
           &  I'   & \\
    -I' &         & \\
           &          &0
\end{array} 
\right), 
\left(
\begin{array}{ccc}
       &  -J' & \\
    J' &     & \\
       &     &0
\end{array} 
\right), 
\left(
\begin{array}{ccc}
J &      &  \\
   & J &   \\
   &     & 0
\end{array} 
\right) 
\mbox{ for (\ref{incl diag su2}), } 
\end{align*}
where  
$J = 
\left(
\begin{array}{cc}
   & -1 \\
 1&     
\end{array}
\right)
$, 
$I' = 
\left(
\begin{array}{cc}
 1&    \\
   & -1    
\end{array}
\right)
$
and 
 $J' = 
\left(
\begin{array}{cc}
 &   1 \\
 1&     
\end{array}
\right).
$

From these computations, we see that $\mathbb{R}^{2} \oplus \mathfrak{su}(2) \not\subset \mathfrak{so}(5)$.
In fact, for (\ref{incl su2=so3}) and (\ref{incl irr su2}), 
we have $\dim_{\mathbb{R}} Z \leq 1$, 
which implies that 
$\mathbb{R}^{2} \oplus \mathfrak{su}(2) \not\subset \mathfrak{so}(5)$. 
For (\ref{incl diag su2}), 
we have $Z \cong \mathfrak{su}(2)$, 
which has no nontrivial commutative Lie subalgebras. 
\end{proof}

\address{Graduate School of Mathematical Sciences, University of Tokyo, 
3-8-1, Komaba, Meguro, Tokyo 153-8914, Japan}
{kkawai@ms.u-tokyo.ac.jp}

\end{document}